\documentclass{amsart}

\usepackage{amsmath,amssymb,amsthm}
\usepackage{pinlabel}

\newtheorem{prop}{Proposition}[section]
\newtheorem{thm}[prop]{Theorem}
\newtheorem{lem}[prop]{Lemma}

\theoremstyle{definition}

\newtheorem{rem}[prop]{Remark}

\newcommand{\tb}{\mathtt{tb}}
\newcommand{\lk}{\mathtt{lk}}
\newcommand{\vlk}{\underline{\lk}}
\newcommand{\rot}{\mathtt{rot}}
\newcommand{\vrot}{\underline{\rot}}
\newcommand{\ttt}{\mathtt{t}}
\newcommand{\ttr}{\mathtt{r}}
\newcommand{\rmd}{\mathrm{d}}

\newcommand{\bfx}{\mathbf{x}}

\newcommand{\N}{\mathbb{N}}
\newcommand{\Q}{\mathbb{Q}}
\newcommand{\R}{\mathbb{R}}
\newcommand{\Z}{\mathbb{Z}}

\newcommand{\SL}{\mathrm{SL}}

\newcommand{\xist}{\xi_{\mathrm{st}}}
\newcommand{\xitight}{\xi_{\mathrm{tight}}}

\begin{document}
	
\author[R.~Chatterjee]{Rima Chatterjee}
\author[H.~Geiges]{Hansj\"org Geiges}
\address{Mathematisches Institut, Universit\"at zu K\"oln,
Weyertal 86--90, 50931 K\"oln, Germany}
\email{rchatt@math.uni-koeln.de}
\email{geiges@math.uni-koeln.de}

\author[S.~Onaran]{Sinem Onaran}
\address{Department of Mathematics, Hacettepe University,
06800 Beytepe-Ankara, Turkey}
\email{sonaran@hacettepe.edu.tr}

\thanks{R.~C.\ and H.~G.\ are partially supported by the SFB/TRR 191
`Symplectic Structures
in Geometry, Algebra and Dynamics', funded by the DFG
(Project-ID 281071066 - TRR 191); S.~O.\
is partially supported by T\"UB\.ITAK Grant No.\ 119F411.
R.~C.\ and S.~O.\ would like to thank the Max Planck Institute for
Mathematics in Bonn for its hospitality.}

\title{Legendrian Hopf links in $L(p,1)$}

\date{}

\begin{abstract}
We classify Legendrian realisations, up to coarse equivalence,
of the Hopf link in the lens spaces
$L(p,1)$ with any contact structure.
\end{abstract}


%\keywords{Legendrian knot, Hopf link, lens space, contact structure, contact
%surgery diagram}

\maketitle

%%%%%%%%%%%%%%%%%%%%%%%%%%%%%%%%%%%%%%%%%%%%%%%%%%%%%%%%%%%%%%%%%%%%%%

\section{Introduction}
\label{section:intro}
By the (positive) \emph{Hopf link} $L_0\sqcup L_1$ in the lens space
$L(p,1)$ we mean the (ordered, oriented) link
formed by the two rational unknots given by the spines of the genus~$1$
Heegaard decomposition, oriented in such a way that their
rational linking number equals $1/p$. By the discussion in
\cite[Section~2]{geon15}, this characterisation determines the link up to
isotopy and a simultaneous change of orientations (which can be effected
by an orientation-preserving diffeomorphism of $L(p,1)$);
the key result in the background is the uniqueness of the
genus~$1$ Heegaard splitting up to isotopy.

The lens space $L(p,1)$ with its natural orientation as a quotient
of $S^3$ can be realised by a single $(-p)$-surgery on an unknot. The Hopf link
in this surgery diagram is shown in Figure~\ref{figure:hl};
the link is positive when both $L_0$ and $L_1$ are oriented as meridians of
the surgery curve in the same way. Indeed, if we label meridian and longitude
(given by the Seifert framing) of the surgery curve by $\mu$ and $\lambda$,
the $(-p)$-surgery amounts to replacing a tubular neighbourhood of
the surgery curve by a solid torus $V_1=S^1\times D^2$, with meridian
$\mu_1=\{*\}\times \partial D^2$ and longitude $\lambda_1=S^1\times\{1\}$
glued as follows:
\[ \mu_1\longmapsto p\mu-\lambda,\;\;\;\lambda_1\longmapsto\mu.\]
Then $L_1=\mu=\lambda_1$ may be thought of as the spine of $V_1$,
and $L_0$ as the spine of the complementary solid torus $V_0$,
with meridian $\mu_0=\lambda$ and longitude $\lambda_0=\mu$.

A Seifert surface $\Sigma_0$
for $pL_0$ is made up of the radial surface in $V_0$ between
the $p$-fold covered spine $pL_0$ and the curve
\[ p\lambda_0-\mu_0=p\mu-\lambda=\mu_1,\]
and a (positively oriented) meridional disc in $V_1$. The spine $L_1$
intersects this disc positively in a single point.

\begin{figure}[h]
\labellist
\small\hair 2pt
\pinlabel $L_1$ [br] at 3 21
\pinlabel $-p$ at  27 22
\pinlabel $L_0$ [bl] at 52 21
\pinlabel $K$ at 27 7
\endlabellist
\includegraphics[scale=2]{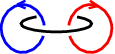}
\caption{The positive Hopf link in $L(p,1)$.}
\label{figure:hl}
\end{figure}
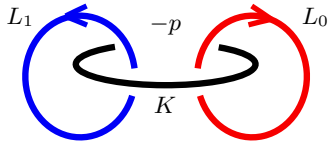

In this paper we extend the classification of Legendrian Hopf links
in $S^3$, see \cite{geon20}, to Legendrian Hopf links in $L(p,1)$
for any $p\in\N$, with any contact structure. For a brief survey
of the rather scant known results on the classification of Legendrian links
(with at least two components) we refer to~\cite{geon20}.
Our result is the first classification of Legendrian links in a
$3$-manifold other than~$S^3$.

Legendrian links in overtwisted contact manifolds are either
\emph{loose}
(the link complement is still overtwisted) or \emph{exceptional}
(the link complement is tight). In the exceptional case, the
link complement may or may not contain Giroux torsion.
In contrast with \cite{geon20}, we only consider the case of vanishing
Giroux torsion (what in \cite{geon20} we called
\emph{strongly exceptional});
the case of Giroux torsion adds numerical complexity
but no significant insight. Beware that the individual components
of an exceptional Legendrian link may well be loose.

Our classification of the Legendrian realisations
(up to coarse equivalence, i.e.\ up to a contactomorphism of the
ambient manifold) of the Hopf link
in $L(p,1)$ is in terms of the rational classical invariants as defined
in~\cite{baet12}.

\begin{lem}
The rational Thurston--Bennequin invariant of the $L_i$, in any contact
structure on $L(p,1)$, is of the form $\tb_{\Q}(L_i)=\ttt_i+\frac{1}{p}$
with $\ttt_i\in\Z$.
\end{lem}

\begin{proof}
By symmetry it suffices to show this for $L_0$, where we can use
the above description of a Seifert surface $\Sigma_0$ for $pL_0$.
The contact framing of $L_0$ is given by a curve $\ttt_0\mu_0+\lambda_0$
with $\ttt_0\in\Z$. With the identifications in the surgery
description of $L(p,1)$ we have
\[ \ttt_0\mu_0+\lambda_0=\ttt_0\lambda+\mu=
\ttt_0(p\lambda_1-\mu_1)+\lambda_1.\]
We can push this curve on $\partial V_1$ a little into $V_1$,
and then the intersection number with the meridional disc bounded by
$\mu_1$ (as part of~$\Sigma_0$) equals $\ttt_0p+1$. To obtain
$\tb_{\Q}(L_0)$, this number has to be divided by~$p$.
\end{proof}

The rational rotation
number $\rot_{\Q}$ is well defined (i.e.\ independent of a choice
of rational Seifert surface), since the Euler class of any
contact structure on $L(p,1)$ is a torsion class.

We also include the $d_3$-invariant of
the overtwisted contact structures. This invariant does not play a direct
role in the classification, but in some cases we need it to determine
whether the link components are loose or exceptional, when we appeal to
the classification of exceptional Legendrian rational unknots in $L(p,1)$
achieved in~\cite{geon15}. The notation $\xi_d$ stands
for an overtwisted contact structure with $d_3=d$. For a complete homotopical
classification of the overtwisted contact structures one also needs
to know their Euler class (or, in the presence of $2$-torsion, a finer
$d_2$-invariant). The Euler class can be computed from
the surgery diagrams we present, using the recipe of
\cite{chke}. These computations, which are rather involved, are omitted here.
In the cases where we use contact cuts to find the Legendrian realisations,
the homotopical classification of the contact structure in question is
more straightforward.

We use the notation $\xitight$ for any one of the tight contact
structures on $L(p,1)$. The homotopical data of the relevant
contact structures containing Legendrian realisations of the Hopf
link can easily be read off from Figure~\ref{figure:case-a}. These tight
structures are Stein fillable and hence have zero Giroux torsion.

The essence of the following main theorem is that the classical invariants
suffice to classify the Legendrian realisations of the Hopf link in
$L(p,1)$, so the Hopf link is what is called Legendrian simple.

\begin{thm}
\label{thm:main}
Up to coarse equivalence, the Legendrian realisations of the
Hopf link in $L(p,1), p\geq 2$, with zero Giroux torsion in
the complement, are as follows. In all cases, the classical
invariants determine the Legendrian realisation.
\begin{itemize}
\item[(a)] In $(L(p,1),\xitight)$ there is a unique realisation for any
combination of classical invariants $(\tb_\Q(L_i),\rot_\Q(L_i))=
(\ttt_i+\frac{1}{p},\ttr_i-\frac{\ttr}{p})$,
$i=0,1$, in the range $\ttt_0,\ttt_1<0$ and
\[\ttr\in\{-p+2,-p+4,\ldots,p-4,p-2\},\]
\[ \ttr_i\in\{\ttt_i+1,\ttt_i+3,\ldots,-\ttt_i-3,-\ttt_i-1\}.\]
For fixed values of $\ttt_0,\ttt_1<0$ this gives a total
of $\ttt_0\ttt_1(p-1)$ realisations.
\item[(b)] For $\ttt_0=0$ and $\ttt_1\leq 0$ there are
$|\ttt_1-1|$ exceptional realisations, all living in
$(L(p,1),\xi_d)$ with $d=(3-p)/4$, made up of an exceptional
component $L_0$ with classical invariants $(\tb_\Q(L_0),\rot_\Q(L_0))=
(\frac{1}{p}, 0)$, and a loose component $L_1$ with
invariants $(\tb_\Q(L_1),\rot_\Q(L_1))=(\ttt_1+\frac{1}{p},\ttr_1)$, where
\[ \ttr_1\in\{\ttt_1,\ttt_1+2,\ldots,-\ttt_1-2,-\ttt_1\}.\]
\item[(c)] For $\ttt_0=0$ and $\ttt_1>0$ the exceptional
realisations are as follows; in all cases both components are loose.
\begin{itemize}
\item[(c1)] For $\ttt_1=1$ there are two realisations,
with classical invariants
\[ (\tb_{\Q}(L_0),\rot_{\Q}(L_0))=\Bigl(\frac{1}{p},\pm\frac{2}{p}\Bigr)\]
and
\[ (\tb_{\Q}(L_1),\rot_{\Q}(L_1))=\Bigl(1+\frac{1}{p},
\pm\Bigl(1+\frac{2}{p}\Bigr)\Bigr).\]
They live in $(L(p,1),\xi_d)$ with $d=(3p-p^2-4)/4p$.
\item[(c2)] For $\ttt_1=2$, there are three exceptional
realisations. Two of them live in $(L(p,1),\xi_d)$,
where $d=\frac{3p-p^2-4}{4p}$, and have classical invariants
$(\tb_\Q(L_0),\rot_\Q(L_0))$ as in (c1)
and
\[ (\tb_\Q(L_1),\rot_\Q(L_1))=\Bigl(2+\frac{1}{p},\pm\Bigl(2+\frac{2}{p})
\Bigr)\Bigr).\]
The third one lives in  $(L(p,1),\xi_d)$ with $d=\frac{7-p}{4}$,
and the invariants are $(\tb_\Q(L_0),\rot_\Q(L_0))=(\frac{1}{p},0)$
and $(\tb_\Q(L_1),\rot_\Q(L_1))=(2+\frac{1}{p},0)$.			
\item[(c3)] For $\ttt_1>2$, there are four exceptional
realisations. The classical invariants are listed in Table
\ref{table:c3} in Section~\ref{section:exceptional}.
\end{itemize}	
\item[(d)] For $\ttt_0,\ttt_1>0$ the exceptional
realisations, with both components loose, are as follows:
\begin{itemize}
\item[(d1)] For $\ttt_0=\ttt_1=1$, there are exactly $p+3$
exceptional realisations. They live in $(L(p,1),\xi_d)$,
where $d=\frac{7p-\ttr^2}{4p}$. They have classical invariants
\[ (\tb_{\Q}(L_0),\rot_{\Q}(L_0))=
\Bigl(1+\frac{1}{p},\frac{\ttr}{p}\Bigr)\]
and
\[ (\tb_\mathbb{Q}(L_1),\rot_\mathbb{Q}(L_1))
=\Bigl(1+\frac{1}{p},\frac{\ttr}{p}\Bigr),\]
where
\[\ttr\in\{-p-2,-p,\ldots, p,p+2\}.\]
\item[(d2)] For $\ttt_0=1$ and $\ttt_1>1$, there are $2(p+2)$ exceptional
realisations, whose classical invariants are given in Table \ref{table:d2}.
\item[(d3)]For $\ttt_0>1$ and $\ttt_1>1$, there are $4(p+1)$ exceptional
realisations, whose classical invariants are given in
Table~\ref{table:d3}.
\end{itemize}	
\item[(e)] For $\ttt_0<0$ and $\ttt_1>0$ the exceptional
realisations are as follows. Here $L_0$ is loose; $L_1$ is exceptional.
\begin{itemize}
\item[(e1)] For $\ttt_1=1$, there are exactly
$|\ttt_0|(p+1)$ exceptional realisations. They live in
$(L(p,1),\xi_d)$ where $d=(3p-\ttr^2)/4p$, where
\[\ttr\in\{-p, -p+2,\ldots, p-2, p\}.\]
The classical invariants are
\[ (\tb_{\Q}(L_0),\rot_{\Q}(L_0))=
\Bigl(\ttt_0+\frac{1}{p},\ttr_0-\frac{\ttr}{p}\Bigr)\]
and
\[ (\tb_\mathbb{Q}(L_1),\rot_\mathbb{Q}(L_1))=
\Bigl(1+\frac{1}{p},-\frac{\ttr}{p}\Bigr),\]
where 
\[\ttr_0\in\{\ttt_0+1,\ttt_0+3,\ldots, -\ttt_0-3, -\ttt_0-1\}.\]
\item[(e2)] For $\ttt_0<0$ and $\ttt_1>1$, there are $2|\ttt_0|p$ exceptional
realisations, whose classical invariants are given in
Section~\ref{subsubsection:e2}.
\end{itemize}
\end{itemize}		
\end{thm}

The proof of Theorem~\ref{thm:main} largely follows the strategy
used in~\cite{geon20} for Legendrian Hopf links in~$S^3$:
find an upper bound on the number of exceptional realisations by enumerating
the tight contact structures on the link complement, and then
show that this bound is attained by giving explicit realisations.

Most of these explicit realisations are in terms
of surgery diagrams, but as in \cite{geon20} there is a
case where a surgery presentation eludes us, and we have to use
contact cuts instead. This case (c1), which is being treated in
Section~\ref{section:cut}, contains most of the conceptually novel aspects
in the present paper. In contrast with \cite{geon20}, we no longer
have a global frame for the contact structure; therefore, the computation
of rotation numbers requires the explicit description of rational
Seifert surfaces, and frames over them, in the context of topological cuts.
The discussion in Section~\ref{section:cut} should prove
useful in a more general analysis of the contact topology of lens
spaces via contact cuts.

\begin{rem}
One has to be a little careful when comparing this result with
the classification of Legendrian Hopf links in $S^3$ in~\cite{geon20}.
For instance, in case (c1), the contact cut description we use
in Section~\ref{section:cut}
corresponds for $p=1$ to the interpretation of $S^3$ as lens space
$L(1,1)$, whereas in \cite{geon20} we read $S^3$ as $L(1,0)$. In case~(a),
the surgery diagram for $S^3$ is empty, and the discussion in the present
paper only makes sense for $p\geq 2$. In most other cases, however, one
obtains the correct results for $S^3$ by allowing $p=1$ in
Theorem~\ref{thm:main}.
\end{rem}
\section{Upper bound for exceptional realisations}
In this section we determine the number of tight contact structures on the
complement of a Legendrian Hopf link $L_0\sqcup L_1$ in
$L(p,1)$, in terms of the
Thurston--Bennequin invariant of the link components.
We start with $S^3$, decomposed into
two solid tori $V_0,V_1$ forming a Hopf link (in the traditional
sense), and a thickened torus $T^2\times[0,1]$, i.e.\
\[ S^3=V_0\cup_{\partial V_0=T^2\times\{0\}} T^2\times [0,1]
\cup_{T^2\times\{1\}=\partial V_1} V_1.\]
Write $\mu_i,\lambda_i$ for meridian and longitude
on $\partial V_i$, and define the gluing in the decomposition above by
\begin{eqnarray*}
	\mu_0     & = & S^1\times\{*\}\times\{0\},\\
	\lambda_0 & = & \{*\}\times S^1\times\{0\},\\
	\mu_1     & = & \{*\}\times S^1\times\{1\},\\
	\lambda_1 & = & S^1\times\{*\}\times\{1\}.
\end{eqnarray*} 

As in the introduction, we think of the Hopf link in $L(p,1)$
as being obtained by $(-p)$-surgery along the spine of~$V_1$.
Slightly changing the notation from the introduction, we write
$\mu_1',\lambda_1'$ for meridian and longitude of the solid torus
$V_1'$ reglued in place of $V_1$, so that the gluing
prescription becomes
\[ \mu_1'\longmapsto p\mu_1-\lambda_1,\;\;\;\lambda_1'\longmapsto \mu_1.\]

Given a Legendrian Hopf link $L_0\sqcup L_1$ in $L(p,1)$
with $\tb_{\Q}(L_i)=\ttt_i+\frac{1}{p}$, we can choose $V_0,V_1'$
(\emph{sic}!) as standard neighbourhoods of $L_0,L_1$, respectively.
This means that $\partial V_0$ is a convex surface with two dividing
curves of slope $1/\ttt_0$ with respect to $(\mu_0,\lambda_0)$;
the slope of $\partial V_1'$ is $1/\ttt_1$ with respect to
$(\mu_1',\lambda_1')$.

Now, on $T^2\times[0,1]$ we measure slopes on the $T^2$-factor with
respect to $(\mu_0,\lambda_0)$.
So we are dealing with a contact structure
on $T^2\times[0,1]$ with convex boundary, two dividing curves
on either boundary component, of slope $s_0=1/\ttt_0$ on
$T^2\times\{0\}$, and of slope $s_1=-p-1/\ttt_1$ on $T^2\times\{1\}$, since
\[ \ttt_1\mu_1'+\lambda_1'=\ttt_1(p\lambda_0-\mu_0)+\lambda_0=
(\ttt_1p+1)\lambda_0-\ttt_1\mu_0.\]
Recall that a contact structure on $T^2\times[0,1]$
with these boundary conditions is called \emph{minimally twisting}
if every convex torus parallel to the boundary has slope between $s_1$
and~$s_0$.

The following proposition covers all possible pairs $(\ttt_0,\ttt_1)$,
possibly after exchanging the roles of $L_0$ and $L_1$.

\begin{prop}
\label{prop:complement}
Up to an isotopy fixing the boundary, the number $N=N(\ttt_0,\ttt_1)$
of tight, minimally twisting contact structures
on $T^2\times[0,1]$ with convex boundary, two dividing curves
on either boundary component of slope $s_0=1/\ttt_0$ and
$s_1=-p-1/\ttt_1$, respectively, is as follows.
\begin{itemize}
\item[(a)] If $\ttt_0,\ttt_1<0$, we have $N=\ttt_0\ttt_1(p-1)$.
\item[(b)] If $\ttt_0=0$ and $\ttt_1\leq 0$, then $N=|\ttt_1-1|$.
\item[(c)] If $\ttt_0=0$, $\ttt_1\geq 1$:
\begin{itemize}
\item[(c1)] $N(0,1)=2$.
\item[(c2)] $N(0,2)=3$.
\item[(c3)] For all $\ttt_1>2$, we have $N(0,\ttt_1)=4$.
\end{itemize}
\item[(d)] If $\ttt_0,\ttt_1>0$:
\begin{itemize}
\item[(d1)] $N(1,1)=p+3$.
\item[(d2)] For all $\ttt_1>1$, we have $N(1,\ttt_1)=2(p+2)$.
\item[(d3)] For all $\ttt_0,\ttt_1>1$, we have $N=4(p+1)$.
\end{itemize}	
\item[(e)] If $\ttt_0<0$, $\ttt_1>0$:
\begin{itemize}
\item[(e1)] For all $\ttt_0<0$ and $\ttt_1>1$, we have
$N=2|\ttt_0|p$.
\item[(e2)] For all $\ttt_0<0$, we have
$N(\ttt_0,1)=|\ttt_0|(p+1)$.
\end{itemize}
\end{itemize}
\end{prop}

\begin{proof}
So that we can use the classification of tight contact structures on
$T^2\times [0,1]$ due to Giroux~\cite{giro00} and Honda~\cite{hond00I},
we normalise the slopes by applying
an element of $\mathrm{Diff}^+(T^2)\cong\SL(2,\Z)$ to $T^2\times[0,1]$
such that the slope on $T^2\times\{0\}$ becomes $s_0'=-1$,
and on $T^2\times\{1\}$ we have $s_1'\leq -1$. If $s_1'<-1$,
the number $N$ is found from a continuous fraction expansion
\[ s_1'=r_0-\cfrac{1}{r_1-\cfrac{1}{r_2-\cdots-\cfrac{1}{r_k}}}
=:[r_0,\ldots,r_k]\]
with all $r_i<-1$ as
\begin{equation}
\label{eqn:N}
N=|(r_0+1)\cdots(r_{k-1}+1)r_k|,
\end{equation}
see~\cite[Theorem~2.2(2)]{hond00I}.
The vector $\begin{pmatrix}x\\y\end{pmatrix}$ stands for
the curve $x\mu_0+y\lambda_0$, with slope $y/x$.
\subsection{Case (a)}
We have 
\[ \begin{pmatrix}
0 & -1\\
1 & -\ttt_0+1 
\end{pmatrix}
\begin{pmatrix}
\ttt_0\\
1
\end{pmatrix}
=\begin{pmatrix}
-1\\
1
\end{pmatrix}\]
and 
\[ \begin{pmatrix}
0 & -1\\
1 & -\ttt_0+1 
\end{pmatrix}
\begin{pmatrix}
-\ttt_1\\
p\ttt_1+1
\end{pmatrix}
=\begin{pmatrix}
-p\ttt_1-1\\
-\ttt_1-p\ttt_0\ttt_1+p\ttt_1-\ttt_0+1
\end{pmatrix},\]
which implies
\[s_1'=\ttt_0-1+\frac{\ttt_1}{p\ttt_1+1}=
\ttt_0-1-\cfrac{1}{-p-\cfrac{1}{\ttt_1}}.\]
For $\ttt_1<-1$ we read this as $[\ttt_0-1,-p,\ttt_1]$; for
$\ttt_1=-1$ and $p>2$, as $[\ttt_0-1,-p+1]$; for $\ttt_1=-1$
and $p=2$, as $[\ttt_0]$.
In all three cases this gives  $N=|\ttt_0\ttt_1(p-1)|$.
\subsection{Case (b)}
For $\ttt_0=0$ and $\ttt_1\leq 0$ we use the transformation
\[ \begin{pmatrix}
-p & -1\\
p+1 & 1 
\end{pmatrix}
\begin{pmatrix}
0\\
1
\end{pmatrix}
=\begin{pmatrix}
-1\\
1
\end{pmatrix}\]
and
\[
\begin{pmatrix}
-p & -1\\
p+1 & 1 
\end{pmatrix}
\begin{pmatrix}
-\ttt_1\\
p\ttt_1+1
\end{pmatrix}
=\begin{pmatrix}
-1\\
1-\ttt_1
\end{pmatrix}.\]
This gives $s_1'=-1+\ttt_1$, whence $N=|\ttt_1-1|$.
\subsection{Case (c)} For $t_0=0$ and $t_1\geq 2$ we have
\[ \begin{pmatrix}
-(p+1) & -1\\
p+2 & 1 
\end{pmatrix}
\begin{pmatrix}
0\\
1
\end{pmatrix}
=\begin{pmatrix}
-1\\
1
\end{pmatrix}\]
and
\[ \begin{pmatrix}
-(p+1) & -1\\
p+2 & 1 
\end{pmatrix}
\begin{pmatrix}
-\ttt_1\\
p\ttt_1+1
\end{pmatrix}
=\begin{pmatrix}
\ttt_1-1\\
-2\ttt_1+1
\end{pmatrix}.\]
A continued fraction expansion for $s_1'=(-2\ttt_1+1)/(\ttt_1-1)$ is given by
\[ [-3,\underbrace{-2,-2,\ldots, -2}_{\ttt_1-2}].\]
Hence, for $\ttt_1>2$ we get 
$N=|(-2)(-1)\cdots (-1)(-2)|=4$; for $\ttt_1=2$ we have $N=3$.

For $\ttt_1=1$ we work instead with the transformation
\[ \begin{pmatrix}
-2p & -1\\
1+2p & 1 
\end{pmatrix}
\begin{pmatrix}
0\\
1
\end{pmatrix}
=\begin{pmatrix}
-1\\
1
\end{pmatrix}\]
and
\[ \begin{pmatrix}
-2p & -1\\
1+2p & 1 
\end{pmatrix}
\begin{pmatrix}
-1\\
p+1
\end{pmatrix}
=\begin{pmatrix}
p-1\\
-p
\end{pmatrix},\]
which gives
\[ s_1'=-\frac{p}{p-1}=
[\underbrace{-2,-2,\ldots, -2}_{p-1}],\]
whence $N=2$.
\subsection{Case (d)}
For $\ttt_0>0$ and $\ttt_1>0$ we compute
\[ \begin{pmatrix}
-p & -1+p\ttt_0\\
1+p & 1-(1+p)\ttt_0 
\end{pmatrix}
\begin{pmatrix}
\ttt_0\\
1
\end{pmatrix}
=\begin{pmatrix}
-1\\
1
\end{pmatrix}\]
and 
\[ \begin{pmatrix}
-p & -1+p\ttt_0\\
1+p & 1-(1+p)\ttt_0 
\end{pmatrix}
\begin{pmatrix}
-\ttt_1\\
p\ttt_1+1
\end{pmatrix}
=\begin{pmatrix}
-1+p\ttt_0+p^2\ttt_0\ttt_1\\
1-\ttt_0-\ttt_1-p\ttt_0-p\ttt_0\ttt_1-p^2\ttt_0\ttt_1
\end{pmatrix}.\]
Hence,
\[s_1'=-1-\frac{\ttt_0+\ttt_1+p\ttt_0\ttt_1}{-1+p\ttt_0+p^2\ttt_0\ttt_1}.\]

For $\ttt_0=1$, the continued fraction expansion is
\[ s_1'=[\underbrace{-2,-2,\ldots, -2}_{p-1}, -(p+3),
\underbrace{-2,-2,\ldots, -2}_{\ttt_1-1}].\]
Thus, $N=p+3$ for $\ttt_1=1$, and $N=2(p+2)$ for $\ttt_1>1$.

For $\ttt_0,\ttt_1>1$ we have the continued fraction expansion
\[ s_1'=[\underbrace{-2,-2,\ldots ,-2}_{p-1}, -3,
\underbrace{-2,-2,\ldots, -2}_{\ttt_0-2}, -(p+2),
\underbrace{-2, -2,\ldots, -2}_{\ttt_1-1}],\]
whence $N=4(p+1)$.
\subsection{Case (e)}
For $\ttt_0<0$ and $\ttt_1>0$ we use the transformation
\[ \begin{pmatrix}
-1 & \ttt_0-1\\
2 & 1-2\ttt_0
\end{pmatrix}
\begin{pmatrix}
\ttt_0\\
1
\end{pmatrix}
=\begin{pmatrix}
-1\\
1
\end{pmatrix}\]
and
\[ \begin{pmatrix}
-1 & \ttt_0-1\\
2 & 1-2\ttt_0
\end{pmatrix}
\begin{pmatrix}
-\ttt_1\\
p\ttt_1+1
\end{pmatrix}
=\begin{pmatrix}
-1+\ttt_0+\ttt_1-p\ttt_1+p\ttt_0\ttt_1\\
1-2\ttt_0-2\ttt_1+p\ttt_1-2p\ttt_0\ttt_1
\end{pmatrix}.\]
Then
\[s_1'=-2-\frac{1+p\ttt_1}{-1+\ttt_0+\ttt_1-p\ttt_1+p\ttt_0\ttt_1}
=[-2, \ttt_0-1,-(p+1),\underbrace{-2,-2,\ldots,-2}_{\ttt_1-1}].\]
For $\ttt_1>1$ this yields $N=2|\ttt_0|p$; for
$\ttt_1=1$ we get $N=|\ttt_0|(p+1)$.
\end{proof}
\section{Computing the invariants from surgery diagrams}
\label{section:invariants}
Except for the case (c1) discussed in terms of contact cuts in
Section~\ref{section:cut}, we are going to describe the
Legendrian realisations of the Hopf link in $L(p,1)$
as front projections of a Legendrian link in a contact surgery
diagram for $(L(p,1),\xi)$ involving only contact $(\pm 1)$-surgeries.
Here we briefly recall how to
compute the classical invariants from such a presentation;
for more details see \cite[Section~5.2]{geon20}.

Write $M$ for the linking matrix of the surgery diagram,
with the surgery knots $K_1,\ldots,K_n$ given auxiliary orientations,
and $\lk(K_j,K_j)$ equal to the topological surgery framing.
The extended linking matrix of a Legendrian knot $L_i$ in this
surgery presentation is
\[ M_i=\left(\begin{array}{c|ccc}
0            & \lk(L_i,K_1) & \cdots & \lk(L_i,K_n)\\ \hline
\lk(L_i,K_1) &              &        &             \\
\vdots       &              & M      &             \\
\lk(L_i,K_n) &              &        &
\end{array}\right). \]
\subsection{Thurston--Bennequin invariant}
Write $\tb_i$ for the Thurston--Bennequin invariant of $L_i$
as a Legendrian knot in $(S^3,\xist)$, that is, before performing
the contact surgeries. Then, in the surgered contact manifold, one has
\[ \tb_{\Q}(L_i)=\tb_i+\frac{\det M_i}{\det M}.\]
\subsection{Rotation number}
Write $\rot_i$ for the rotation number of $L_i$ before the surgery.
With
\[ \vrot:=\bigl(\rot(K_1),\ldots,\rot(K_n)\bigr)\]
and
\[ \vlk_i:=\bigl(\lk(L_i,K_1),\ldots,\lk(L_i,K_n)\bigr)\]
we have
\[ \rot_{\Q}(L_i)=\rot_i-\langle\vrot,M^{-1}\vlk_i\rangle.\]
\subsection{The $d_3$-invariant}
The surgery diagram describes a $4$-dimensional handlebody $X$ with
signature $\sigma$ and Euler characteristic~$\chi=1+n$. Let
$c\in H^2(X)$ be the cohomology class determined by $c(\Sigma_j)=
\rot(K_j)$, where $\Sigma_j$ is the oriented surface made up
of a Seifert surface for $K_j$ and the core disc of the
corresponding handle. Write $q$ for the number of contact $(+1)$-surgeries.

Then the $d_3$-invariant is given by the formula
\[ d_3(\xi)=\frac{1}{4}(c^2-3\sigma-2\chi)+q,\]
where $c^2$ is computed as follows: find the solution vector $\bfx$
of the equation $M\bfx=\vrot$; then $c^2=\bfx^{\ttt}M\bfx=\langle\bfx,\vrot
\rangle$.

The signature $\sigma$ can be computed from the linking matrix
corresponding to the surgery diagram; more efficiently, one can usually
compute it using Kirby moves as described in \cite[p.~1433]{geon20}.
\section{Hopf links in tight $L(p,1)$}
For given values of $\ttt_0,\ttt_1<0$, we have
$\ttt_0\ttt_1(p-1)$ explicit realisations in $(L(p,1),\xitight)$ as
shown in Figure~\ref{figure:case-a}. Here the numbers $k,k_i$ and
$\ell,\ell_i$ refer to the exterior cusps, so that
the surgery curve $K$ has $\tb=-k-\ell+1$ (and $\tb_i=-k_i-\ell_i+1)$,
and to obtain $L(p,1)$
by a contact $(-1)$-surgery on $K$ we need $k+\ell=p$. The
$\tb_i$ can take any negative value. With $L_0, L_1$ both oriented
clockwise, the Hopf link is positive,
and in this case we have $\rot_i=\ell_i-k_i$.

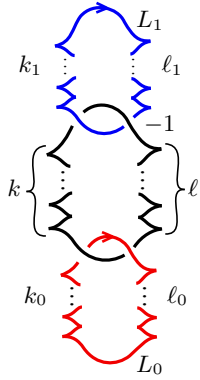
\begin{figure}[h]
\labellist
\small\hair 2pt
\pinlabel $L_0$ [tl] at 40 5
\pinlabel $L_1$ [bl] at 40 125
\pinlabel $k$ [r] at 0 67
\pinlabel $\ell$ [l] at 59 67
\pinlabel $-1$ [bl] at 43 86
\pinlabel $k_1$ [r] at  8 113
\pinlabel $\ell_1$ [l] at 49 113
\pinlabel $k_0$ [r] at 10 26
\pinlabel $\ell_0$ [l] at 51 26
\endlabellist
\centering
\includegraphics[scale=1]{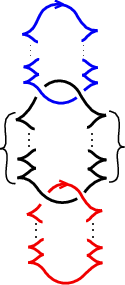}
\caption{Legendrian Hopf links in $(L(p,1),\xitight)$. Here
              $k+\ell=p$.}
\label{figure:case-a}
\end{figure}

The values of the classical invariants as claimed in
Theorem~\ref{thm:main}~(a) now follow easily from the formulas in
Section~\ref{section:invariants}, with $\ttt_i=\tb_i$, $\ttr_i=\rot_i$,
and $\ttr=\rot(K)$. For the $d_3$-invariant we observe
that $\sigma=-1$, $\chi=2$, and $c^2=-\ttr^2/p$, whence
\[ d_3=-\frac{1}{4}\Bigl(1+\frac{\ttr^2}{p}\Bigr).\]

There are no realisations in $(L(p,1),\xitight)$ with one of the
$\ttt_i$ being non-negative, since Legendrian rational unknots
in $(L(p,1),\xitight)$ satisfy $\ttt_i<0$ by
\cite[Theorem~5.5]{baet12} or \cite[Theorem~7.1]{geon15}.

This proves part (a) of Theorem~\ref{thm:main}.

\begin{rem}
\label{rem:tight-d3}
The values of the $d_3$-invariant found above exhausts
all possibilities for the $d_3$-invariant of the tight contact structures
on $L(p,1)$; see \cite[Section~4]{geon15}.
\end{rem}
\section{Legendrian Hopf links via contact cuts}
\label{section:cut}
In this section we use contact cuts to find Legendrian realisations
of Hopf links in $L(p,1)$ for the case (c1).
\subsection{$L(p,1)$ as a cut manifold}
We first want to give a topological description of $L(p,1)$ as a
cut manifold in the sense of Lerman~\cite{lerm01}. We start
from $T^2\times[0,1]=S^1\times S^1\times[0,1]$ with coordinates
$x,y\in S^1=\R/\Z$ and $z\in[0,1]$. Collapsing the first $S^1$-factor
in $S^1\times S^1\times\{0\}$ is equivalent to attaching a solid torus
to  $T^2\times[0,1]$ along $T^2\times\{0\}$ by sending the
meridian of the solid torus to $S^1\times\{*\}\times\{0\}$. This, of course,
simply amounts to attaching a collar to a solid torus, and the meridian
of this `fattened' torus is $\mu_0:=S^1\times\{*\}\times\{1\}$. As longitude
we take the curve $\lambda_0:=\{*\}\times S^1\times\{1\}$.

Now, as described in Section~\ref{section:intro},
$L(p,1)$ is obtained from this
solid torus by attaching another solid torus, whose meridian is glued
to the curve $p\lambda_0-\mu_0$. Equivalently, we may collapse
the foliation of $T^2\times\{1\}$ by circles in the class $p\lambda_0-\mu_0$.

Consider the $p$-fold cover $(\R/\Z)\times(\R/p\Z)\times [0,1]\rightarrow
(\R/\Z)\times(\R/\Z)\times[0,1]$. Set $\tilde{\lambda}_0:=\{*\}\times
(\R/p\Z)\times\{1\}$. The foliation of $(\R/\Z)\times(\R/p\Z)\times [0,1]$
by circles in the class $\tilde{\lambda}_0-\mu_0$ descends to the
foliation defined by $p\lambda_0-\mu_0$. This exhibits $L(p,1)$
as a $\Z_p$-quotient of $L(1,1)=S^3$. Beware that in \cite{geon20}
we used the description of $S^3$ as $L(1,0)$ in the cut construction,
so for $p=1$ one has to take this into account when comparing
the discussion here with the results in~\cite{geon20}.
\subsection{Contact structures from contact cuts}
A contact structure on $T^2\times[0,1]$ will descend to the cut manifold
$L(p,1)$ if, at least near the boundary, it is invariant under the
$S^1$-action whose orbits on the boundary are collapsed to a point,
and if the $S^1$-action is tangent to the contact structure along the
boundary.

We define $a\in (0,\pi/2)$ by the condition $\tan a=p$. For
$\ell\in\N_0$, consider the contact form
\[ \alpha_{\ell}:=\sin\bigl((a+\ell\pi)z\bigr)\,\rmd x+
\cos\bigl((a+\ell\pi)z\bigr)\,\rmd y\]
on $T^2\times[0,1]$. This $1$-form is invariant under the flows of both
$\partial_x$ and $\partial_y$. Along $T^2\times\{0\}$, we have
$\partial_x\in\ker\alpha_{\ell}$; along $T^2\times\{1\}$,
the vector field $-\partial_x+p\partial_y$ is in $\ker\alpha_{\ell}$.
Thus, $\alpha_{\ell}$ descends to a contact form on $L(p,1)$.

By \cite[Proposition~6.1]{geon20}, adapted to the description of $S^3$
as $L(1,1)$, the lift of $\ker\alpha_0$ to $S^3$ is the standard tight contact
structure~$\xist$. Indeed, the map $(x,y)\mapsto (x+y,y)$ sends this lifted
contact structure on $T^2\times [0,1]$, up to isotopy rel boundary and
compatibly with the cuts on the boundary, to the one shown in \cite{geon20}
to define $\xist$ on $S^3=L(1,0)$.

This means that
$\ker\alpha_0$ is the unique (up to diffeomorphism)
universally tight contact structure on $L(p,1)$. The contact structure
$\ker\alpha_{\ell+1}$ is obtained from $\ker\alpha_{\ell}$ by
a $\pi$-Lutz twist. In particular, $\ker\alpha_2$ is obtained from
$\ker\alpha_0$ by a full Lutz twist, and hence is homotopically
equivalent to the universally tight contact structure;
see \cite[Lemma 4.5.3]{geig08}. The surgery description for
the universally tight contact structure on $L(p,1)$ is a contact
$(-1)$-surgery on a $(p-2)$-fold stabilised standard Legendrian
unknot (with $\tb=-1$ and $\rot=0$) in $(S^3,\xist)$, all stabilisations
having the same sign. This contact structure has $d_3$-invariant equal
to $(3p-p^2-4)/4p$ and Euler class $\pm(p-2)$ (in terms of the
natural cohomology generator); see \cite[Section~4]{geon15}
and~\cite{dgs04}.
\subsection{A Legendrian Hopf link in $(L(p,1),\ker\alpha_2)$}
The contact planes in $\ker\alpha_2$ on $T^2\times[0,1]$ have
slope $0$ with respect to $(\partial_x,\partial_y)$ at $z=0$.
As $z$ increases to $z=1$, the contact planes twist (with decreasing slope)
for a little over $2\pi$, until they reach slope $-p$ for the third time
(and after passing through slope $\pm\infty$ twice), similar to
Figure~18 in~\cite{geon20}.

Define $z_0<z_1$ in the interval $[0,1]$ by the conditions
\[ a+2\pi z_0=\frac{\pi}{2}\;\;\;\text{and}\;\;\;
-\sin\bigl((a+2\pi)z_1\bigr)+(p+1)\cos\bigl((a+2\pi)z_1\bigr)=0.\]
This means that at $z=z_0$ the contact planes are vertical for the
first time, and $T^2\times\{z_1\}$ is the second torus (as $z$ increases
from $0$ to~$1$) where the characteristic foliation has slope $-(p+1)$.

Consider the link $L_0\sqcup L_1$ made up of a $(0,1)$-curve
on the torus $\{z=z_0\}$ and a $(-1,p+1)$-curve on $\{z=z_1\}$.
This link is topologically isotopic to the one made up of
a $(0,1)$-curve on $\{z=0\}$ and a $(-1,p+1)$-curve on $\{z=1\}$.
Since these respective curves have intersection number $\pm 1$ with
the curves we collapse at either end,
\[ (0,1)\bullet (1,0)=-1,\;\;\; (-1,p+1)\bullet (-1,p)=1,\]
they are isotopic to the spine of the solid tori attached at either end,
and so they constitute a Hopf link.

A Seifert surface $\Sigma_0$ for $pL_0$, regarded as $p$ times the spine
of the solid torus attached at $z=0$, is given by the following pieces:
\begin{itemize}
\item[-] a helicoidal annulus between $pL_0$ and a $(-1,p)$-curve on $\{z=0\}$,
\item[-] an annulus between this $(-1,p)$-curve on $\{z=0\}$ and the same
curve on $\{z=1\}$,
\item[-] the meridional disc attached to the latter curve at $z=1$.
\end{itemize}
Then
\[ L_1\bullet\Sigma_0=(-1,p+1)\bullet(-1,p)=-p + (p+1)=1.\]
Notice that $\partial_y$ is positively transverse to $\Sigma_0$
on $T^2\times[0,1]$, so our calculation gives the correct sign
of the intersection number. Thus, $L_0\sqcup L_1$ is a \emph{positive}
Hopf link.

Both components $L_0$ and $L_1$ of this Legendrian Hopf link are
loose, since the contact planes twist by more than $\pi$ in either interval
$(z_0,1)$ and $(0,z_1)$, so the cut produces an overtwisted disc.

\begin{lem}
The Legendrian Hopf link $L_0\sqcup L_1$ in $(L(p,1),\ker\alpha_2)$
is exceptional.
\end{lem}

\begin{proof}
Arguing by contradiction, suppose there were an overtwisted disc in the
complement of the link. This disc would persist in the complement of the
transverse link $L_0'\sqcup L_1'$ obtained by pushing $L_0$ a little
towards $z=0$, and $L_1$ towards $z=1$.
In $T^2\times[0,1]$, this link is transversely isotopic
to the link made up of a $(0,1)$-curve on $\{z=0\}$ and a $(-1,p+1)$-curve
on $\{z=1\}$. In the cut manifold $L(p,1)$, this gives a transverse isotopy
(and hence an ambient contact isotopy)
from $L_0'\sqcup L_1'$ to the transverse link made up of the collapsed
boundaries $(T^2\times\{0\})/S^1$ and $(T^2\times\{1\})/S^1$.
The complement of the latter, however, is
contactomorphic to $(T^2\times (0,1),\ker\alpha_2)$, which is tight, as it
embeds into the standard tight contact structure on~$\R^3$.
\end{proof}
\subsection{Computing $\tb_{\Q}$}
A Legendrian push-off of $L_0$ is simply a parallel $(0,1)$-curve
$L_0'$ on $\{ z=z_0\}$. The topological isotopy from $pL_0$ to $p$ times
the spine of the solid torus attached at $z=0$ can be performed in the
complement of $L_0'$. So the rational Thurston--Bennequin invariant
(as defined in~\cite{baet12}) of $L_0$ is given by
\[ \tb_{\Q}(L_0)=\frac{1}{p}L_0'\bullet\Sigma_0=
\frac{1}{p}(0,1)\bullet(-1,p)=\frac{1}{p}.\]

For $L_1$ we argue similarly. A Legendrian push-off $L_1'$ is given by a
parallel $(-1,p+1)$-curve on $\{z=z_1\}$. The isotopy of $L_1$ to the spine
of the solid torus attached at $z=1$ can be performed in the complement
of $L_1'$, so $\tb_{\Q}(L_1)$ can be computed as $L_1'\bullet\Sigma_1$,
where $\Sigma_1$ is a Seifert surface for $p$ times that spine,
made up of the following pieces:
\begin{itemize}
\item[-] a helicoidal annulus between $pL_1$ and a curve on
$\{z=1\}$ in the class
\[ p(-1,p+1)-(p+1)(-1,p)=(1,0)\]
(see below for an explanation of this choice),
\item[-] an annulus between this $(1,0)$-curve on $\{z=1\}$ and the same
curve on $\{z=0\}$,
\item[-] the meridional disc attached to the latter curve at $z=0$.
\end{itemize}
For the first constituent of this Seifert surface, notice that the $p$-fold
covered spine of a solid torus can be joined by an annulus to any
simple curve on the boundary in a class $p\lambda+k\mu$, $k\in\Z$,
where $\lambda$ is any longitude on the boundary, and $\mu$ the meridian.
Or, more directly, simply observe that in our case
$\mu=(-1,p)$, and $(1,0)\bullet (-1,p)=p$.

The vector field $\partial_y$ is positively transverse to $\Sigma_1$
on $T^2\times [0,1]$, so we obtain
\[ \tb_{\Q}(L_1)=\frac{1}{p}L_1'\bullet\Sigma_1=\frac{1}{p}(p+1)
=1+\frac{1}{p}.\]
\subsection{Frames for $\ker\alpha_{\ell}$}
A frame for $\ker\alpha_{\ell}$ on $T^2\times[0,1]$, compatible with the
orientation defined by $\rmd\alpha_{\ell}$, is given by
\[ \partial_z\;\;\;\text{and}\;\;\;
X_{\ell}:=\cos\bigl((a+\ell\pi)z\bigr)\partial_x-
\sin\bigl((a+\ell\pi)z\bigr)\partial_y.\]
This frame does \emph{not} descend to a frame of the contact structure
on $L(p,1)$.

At $z=0$ we have $X_{\ell}=\partial_x$. If we think of the cut at
$z=0$ as being defined by attaching a solid torus, with meridian
being sent to the $x$-curves, the vector field $\partial_z$ is
outward radial along the boundary of the solid torus, and $X_{\ell}=\partial_x$
is positively tangent to the meridional curves. It follows that a frame
of $\ker\alpha_{\ell}$ that extends over the cut at $z=0$ is given by
\[ \cos(2\pi x)\partial_z-\sin(2\pi x)X_{\ell}\;\;\;\text{and}\;\;\;
\sin(2\pi x)\partial_z+\cos(2\pi x)X_{\ell};\]
cf.\ \cite[Figure 16]{geon20}.

At $z=1$, where we collapse the flow lines of $-\partial_x+p\partial_y$, we
have
\[ X_{\ell}=\pm(\cos a\, \partial_x-\sin a\,\partial_y)=
\mp \cos a\, (-\partial_x+p\,\partial_y),\]
depending on $\ell$  being even or odd.
If we think of the cut again as attaching a solid torus, now $\partial_z$
is inward radial along the boundary of the solid torus, and $X_{\ell}$
is tangent to the meridional curve (positively for $\ell$ odd,
negatively for $\ell$ even). So the frame that extends over the cut
at $z=1$ is
\[ \cos(2\pi x)\partial_z+\sin(2\pi x)X_{\ell}\;\;\;\text{and}\;\;\;
-\sin(2\pi x)\partial_z+\cos(2\pi x) X_{\ell}\;\;\;
\text{for $\ell$ even},\]
and
\[ \cos(2\pi x)\partial_z-\sin(2\pi x)X_{\ell}\;\;\;\text{and}\;\;\;
\sin(2\pi x)\partial_z+\cos(2\pi x) X_{\ell}\;\;\;
\text{for $\ell$ odd}.\]
Thus, only for $\ell$ odd do we have a global frame.
\subsection{Computing $\rot_{\Q}$}
We now look again at the Legendrian Hopf link $L_0\sqcup L_1$
in $(L(p,1),\ker\alpha_2)$.

For the computation of the intersection number $L_1\bullet\Sigma_0$
between $L_1$ and a Seifert surface for $pL_0$ we had the freedom
to isotope $L_0$ (topologically) in the complement of
$L_1$ to the spine of the solid torus
attached at $z=0$. When we want to compute $\rot_{\Q}(L_0)$, we need to
work with a Seifert surface for the original Legendrian $L_0$.
This means that we need to work with the Seifert surface $\Sigma_0$
made up of the following pieces:
\begin{itemize}
\item[-] a $p$-fold covered straight annulus between $pL_0$, i.e.\
the $(0,p)$-curve on the torus $\{z=z_0\}$, and the $p$-fold covered spine
of the solid torus attached at $z=0$,
\item[-] a helicoidal annulus between $p$ times the spine
and a $(-1,p)$-curve on the torus $\{z=0\}$,
\item[-] an annulus between this $(-1,p)$-curve on $\{z=0\}$ and the same
curve on $\{z=1\}$,
\item[-] the meridional disc attached to the latter curve at $z=1$.
\end{itemize}

\begin{rem}
This surface is not embedded, but the computation of rotation numbers
is homological, so this is not a problem.
\end{rem}

Our aim is to find a frame of $\ker\alpha_2$ defined over~$\Sigma_0$.
We begin with the frame defined over $T^2\times[0,1]$
by $\cos(2\pi x)\partial_z+\sin(2\pi x)X_2$
(and its companion defining the correct orientation), which is the
one that extends over the cut at $z=1$. In particular, this frame
is defined over the third and fourth constituent of $\Sigma_0$, and we need
to extend it over the first two pieces of~$\Sigma_0$.

Write $(\R/\Z)\times D^2$ with coordinates $(y;r,\theta)$ for the solid torus
attached at $z=0$. We pass to the $p$-fold covers
\[ (\R/\Z)^2\times [0,1]\longrightarrow (\R/\Z)^2\times[0,1],\;\;
(x,y,z)\longmapsto (x,py,z)\]
and
\[ (\R/\Z)\times D^2\longrightarrow (\R/\Z)\times D^2,\;\;
(y;r,\theta)\longmapsto (py;r,\theta),\]
where the lifted pieces of $\Sigma_0$ are embedded:
a straight annulus between the $(0,1)$-curve on $\{z=z_0\}$
and the spine of the solid torus, plus a helicoidal annulus between
the spine and the $(-1,1)$-curve on the boundary of the solid torus.
For the following homotopical considerations, we may think
of $\ker\alpha_2$ as being extended over that solid torus
as the constant horizontal plane field.

Along the $(-1,1)$ curve on $\{z=0\}$, parametrised as $\R/\Z\ni t\mapsto
(x(t),y(t),0)=(-t,t,0)$, the frame we are considering is
\[ \cos(2\pi t)\partial_z-\sin(2\pi t)\partial_x.\]
In the cylindrical coordinates $(y;r,\theta)$ on the solid torus
this translates into the frame
\[ \cos(2\pi t)\partial_r-\sin(2\pi t)\partial_{\theta}\]
along the curve $t\mapsto (y(t),\theta(t))=(t,-2\pi t)$ on $\{r=1\}$.

Next, we translate this into Cartesian coordinates $(u,v)$ on
the $D^2$-factor. With $r\partial_r=u\partial_u+v\partial_v$
and $\partial_{\theta}=u\partial_v-v\partial_u$, and the curve in question
being $(u(t),v(t))=(\cos 2\pi t,-\sin 2\pi t)$, this gives the frame
\[ \begin{array}{rcl}
\cos(2\pi t)\bigl(\cos(2\pi t)\partial_u-\sin(2\pi t)\partial_v\bigr)-&&\\[2mm]
\sin(2\pi t)\bigl(\cos(2\pi t)\partial_v+\sin(2\pi t)\partial_u\bigr) &=&
\cos(4\pi t)\partial_u-\sin(4\pi t)\partial_v.
\end{array}\]

This formula defines the extension of the frame over the helicoidal annulus
and the part of the straight annulus inside the solid torus.
The intersection of the straight annulus with the torus $\{z=0\}$
is the $(0,1)$-curve, parametrised as $t\mapsto (x(t),y(t))=(0,t)$,
and if we take this curve to be given by $\{ u=1,v=0\}$ (so that
$\partial_u=\partial_z$ and $\partial_v=\partial_x$), the frame is now
written as
\[ \cos(4\pi t)\partial_z-\sin(4\pi t)\partial_x,\]
which extends over the annulus between the $(0,1)$-curve on $\{z=0\}$
and that on $\{z=z_0\}$ (i.e.\ the $p$-fold cover of~$L_0$) as
\[ \cos(4\pi t)\partial_z-\sin(4\pi t)X_2.\]
Notice that at $z=z_0$ we have $X_2=-\partial_y$. The
orientation of $\ker\alpha_2$ is defined by $(\partial_z,X_2)$, so
the frame makes two \emph{negative} rotations with respect to the
tangent vector $\partial_y=-X_2$ of the $p$-fold
covered~$L_0$. We conclude that $\rot_{\Q}(L_0)=2/p$.

\vspace{1mm}

Next we compute $\rot_{\Q}(L_1)$ in an analogous fashion. We now use the frame
$\cos(2\pi x)\partial_z-\sin(2\pi x)X_2$ on $T^2\times [0,1]$, which is the
one that extends over the cut at $z=0$. The cut we perform at $z=1$
corresponds to the attaching of a solid torus $(\R/\Z)\times D^2$
using the gluing map
\[ \mu=\{0\}\times\partial D^2\longmapsto (-1,p)\;\;\;\text{and}\;\;\;
\lambda=(\R/\Z)\times\{1\} \longmapsto (-1,p+1).\]
Note that with respect to the orientation defined by $(\partial_x,\partial_y)$,
the intersection number of meridian and longitude is
\[ \mu\bullet\lambda=(-1,p)\bullet(-1,p+1)=-1,\]
which is what we want, since $\partial_z$ is the inward normal
of the solid torus along its boundary, so this boundary is oriented
by $(\partial_y,\partial_x)$. Notice also that the tangent direction
of $\mu$ coincides with $-X_2$.

The Legendrian knot $L_1$ on $\{z=z_1\}$ is a $(-1,p+1)$-curve, so the
parallel curve on $\{z=1\}$ is in the class of~$\lambda$. The relevant parts
of the Seifert surface $\Sigma_1$ for $pL_1$ in the $p$-fold cover, i.e.\
the lift with respect to the map
\[ (\R/\Z)\times D^2\longrightarrow (\R/\Z)\times D^2,\;\;
(s;r,\theta)\longmapsto (ps;r,\theta),\]
are the following:
\begin{itemize}
\item[-] a straight annulus between the lifted longitude and the spine,
\item[-] a helicoidal annulus between the spine and the $(1,0)$-curve on
$\{z=1\}$; recall that $p\lambda-(p+1)\mu=(1,0)$.
\end{itemize}

At $z=1$, the frame $\cos(2\pi x)\partial_z-\sin(2\pi x)X_2$
we have chosen equals
\[ \cos(2\pi x)\partial_z+\sin(2\pi x)\cos a(-\partial_x+p\partial_y).\]
Along the $(1,0)$-curve, parametrised as $t\mapsto (x(t),y(t))=(t,0)$,
this frame is
\[ \cos(2\pi t)\partial_z+\sin(2\pi t)\cos a(-\partial_x+p\partial_y)=
-\cos(2\pi t)\partial_r+\sin(2\pi t)\partial_{\theta}.\]
Translated into Cartesian coordinates $(u,v)$ on the $D^2$-factor of
the solid torus, the curve (or rather its projection to
the $D^2$-factor) becomes
\[ (u(t),v(t))=(\cos 2\pi(p+1)t,-\sin 2\pi(p+1)t),\]
and translating the frame into Cartesian coordinates as above, we obtain
\[ \begin{array}{rc}
-\cos(2\pi t)\bigl(\cos(2\pi(p+1)t)\partial_u-
\sin(2\pi(p+1)t)\partial_v\bigr)+&\\[2mm]
\sin(2\pi t)\bigl(\cos(2\pi(p+1)t)\partial_v+
\sin(2\pi(p+1)t)\partial_u\bigr)& = \\[2mm]
-\cos(2\pi(p+2)t)\partial_u+\sin(2\pi(p+2)t)\partial_v.&
\end{array} \]
This formula defines the extension of the frame over the helicoidal
annulus and the straight annulus inside the solid torus.
Along the intersection of the straight annulus with
the boundary of the solid torus, given again by $\{u=1,v=0\}$,
and further on the annulus between $p\lambda=p(-1,p+1)$ and $pL_1$,
this frame extends as
\[ \cos(2\pi(p+2)t)\partial_z-\sin(2\pi(p+2)t)X_2,\]
since at $(u,v)=(1,0)$ the vector $\partial_u$ is the
outward normal $-\partial_z$, and $\partial_v$ points in
meridional direction, which is identified with $-X_2$.

This frame makes $p+2$ \emph{negative} rotations with respect
to the tangent direction $X_2$ of $L_1$, which yields
$\rot_{\Q}(L_1)=(p+2)/p=1+2/p$.

\vspace{1mm}

Since we may flip the orientations of $L_0$ and $L_1$ simultaneously,
this gives us in total two realisations, with invariants
\[ (\tb_{\Q}(L_0),\rot_{\Q}(L_0))=\Bigl(\frac{1}{p},\pm\frac{2}{p}\Bigr),
\;\;\;
(\tb_{\Q}(L_1),\rot_{\Q}(L_1))=\Bigl(1+\frac{1}{p},
\pm\Bigl(1+\frac{2}{p}\Bigr)\Bigr).\]

\begin{rem}
For $p=1$, this accords with the case $(1,\pm 2)$ and $(2,\pm 3)$
discussed in \cite[Section~7.4]{geon20}.
\end{rem}
\section{Detecting exceptional links}
\label{section:detecting}
Before we turn to the Legendrian realisations of Legendrian
Hopf links in $L(p,1)$ in terms of contact surgery diagrams, we discuss
how to establish that a given Hopf link is exceptional, and how to decide
whether the individual components are loose or exceptional.

First one needs to verify that the contact structure given
by the surgery diagram is overtwisted.
If the $d_3$-invariant differs from that of any of the tight
structure (see Remark~\ref{rem:tight-d3}), this
is obvious. If the $d_3$-invariant \emph{does} match that of a tight
structure, overtwistedness can be shown by exhibiting
a Legendrian knot in the surgered manifold that violates the
Bennequin inequality \cite[Theorem~1]{baet12} for Legendrian knots
in tight contact $3$-manifolds. In all our examples, one of $L_0$ or $L_1$
will have this property. Alternatively, one can appeal to
the classification of Legendrian rational unknots in $L(p,1)$ with
a tight contact structure~\cite[Theorem~7.1~(a)]{geon15}. If the invariants
of $L_0$ or $L_1$ do not match those listed there (in particular,
$\tb_{\Q}(L)=\ttt+1/p$ with $\ttt$ a negative integer),
then the contact structure
must be overtwisted. Again, this covers all cases (b)--(e).

Secondly, we need to establish that the contact structure on the
link complement $L(p,1)\setminus(L_0\sqcup L_1)$ is tight. The method we use is
to perform contact surgeries on $L_0$ and $L_1$, perhaps also on
Legendrian push-offs of these knots, such that the
resulting contact manifold is tight. If there had been an overtwisted
disc in the complement of $L_0\sqcup L_1$, this would persist
after the surgery.

Here we rely on the cancellation lemma from \cite{dige04},
cf.~\cite[Proposition 6.4.5]{geig08}, which says that a contact
$(-1)$-surgery and a contact $(+1)$-surgery along a Legendrian knot
and its Legendrian push-off, respectively, cancel each other.
For instance, if by contact $(-1)$-surgeries on $L_0$ and $L_1$
we can cancel all contact $(+1)$-surgeries in the surgery diagram,
and thus obtain a Stein fillable and hence tight contact $3$-manifold,
the Legendrian Hopf link will have been exceptional.

To determine whether one of the link components is loose,
we sometimes rely on the classification of exceptional rational unknots
in $L(p,1)$ given in \cite[Theorem~7.1]{geon15}:

\begin{thm}
\label{thm:geon15}
Up to coarse equivalence, the exceptional rational unknots in
$L(p,1)$ are classified by their classical invariants
$\tb_{\Q}$ and $\rot_{\Q}$. The possible values of
$\tb_\mathbb{Q}$ are $\ttt+1/p$ with $\ttt\in\N_0$. For $\ttt=0$, there is a
single exceptional knot, with $\rot_{\Q}=0.$ For $\ttt=1$, there are $p+1$
exceptional knots, with
\[ \rot_{\Q}\in\Bigl\{-1, -1+\frac{2}{p}, -1+\frac{4}{p}, \ldots,
-1+\frac{2p}{p}\Bigr\}.\]
For $\ttt\geq 2$, there are $2p$ exceptional knots, with
\[\rot_{\Q}\in\Bigl\{\pm\Bigl(\ttt-2+\frac{2}{p}\Bigr),
\pm\Bigl(\ttt-2+\frac{4}{p}\Bigr),\ldots,
\pm\Bigl(\ttt-2+\frac{2p}{p}\Bigr)\Bigr\}.\]
\end{thm}

In a few cases, the invariants of one link component
equal those realised by an exceptional rational unknot,
but we detect looseness by computing the $d_3$-invariant
and observing, again comparing with~\cite{geon15},
that it does not match that of an exceptional realisation.
\section{Exceptional Hopf links}
\label{section:exceptional}
In this section we find exceptional Legendrian realisations of the
Hopf link, except case (c1)---which is covered by
Section~\ref{section:cut}---, in contact surgery diagrams for $L(p,1)$.
This completes the proof of Theorem~\ref{thm:main}.
\subsection{Kirby diagrams}
We begin with some examples of Kirby diagrams of the Hopf link
that will be relevant in several cases of this classification.
The proof of the following lemma is given by the Kirby moves in the
corresponding diagrams.

\begin{lem}
(i) The oriented link $L_0\sqcup L_1$ in the surgery diagram
shown in the first line of Figure~\ref{figure:kirby-c3}
is a positive or negative Hopf link in~$L(p,1)$, depending
on $k$ being even or odd.
The same is true for the links shown
in the first line of Figure~\ref{figure:kirby-d2} and~\ref{figure:kirby-e2},
respectively. 

(ii) The oriented link $L_0\sqcup L_1$ shown in Figure~\ref{figure:kirby-d1}
is a positive Hopf link in $L(p,1)$.

(iii) The oriented link $L_0\sqcup L_1$ in the first line
of Figure~\ref{figure:kirby-d3}
is a positive or negative Hopf link depending on $k_0$ and $k_1$
having the same parity or not.\qed
\end{lem}

\begin{figure}[h]
\labellist
\tiny\hair 2pt
\pinlabel $L_1$ at 10 266
\pinlabel $-1$ at 23 256
\pinlabel $-2$ at 11 248
\pinlabel $k$ [t] at 65 228
\pinlabel $-2$ at 40 260
\pinlabel $-2$ at 51 260
\pinlabel $-2$ at 84 261
\pinlabel $-3$ at 100 261
\pinlabel $-1$ at 116 247
\pinlabel ${0/0/\cdots /0}$ [t] at 129 230
\pinlabel $L_0$ at 143 266
\pinlabel $L_1$ at 5 210
\pinlabel $1$ [bl] at 18 185
\pinlabel $1$ [r] at 0 200
\pinlabel $1$ [r] at 0 193
\pinlabel $-2$ at 38 206
\pinlabel $-2$ at 51 206
\pinlabel $-2$ at 84 206
\pinlabel $-2$ at 100 206
\pinlabel $1$ at 112 184
\pinlabel $k$ at 65 172
\pinlabel $L_0$ at 149 208
\pinlabel ${1/1/\cdots /1}$ [t] at 129 177
\pinlabel $L_1$ at 5 158
\pinlabel $1$ at 24 163
\pinlabel $1$ [r] at 6 148
\pinlabel $1$ [r] at 6 141
\pinlabel $-2$ at 38 159
\pinlabel $-2$ at 51 159
\pinlabel $-2$ at 82 159
\pinlabel $-2$ at 98 159
\pinlabel $-p-1$ at 113 159
\pinlabel $k$ at 75 120
\pinlabel $L_0$ at 144 152
\pinlabel $L_1$ at 2 105
\pinlabel $-1$ at 23 115
\pinlabel $-2$ at 38 111
\pinlabel $-2$ at 51 111
\pinlabel $-2$ at 82 111
\pinlabel $-2$ at 98 111
\pinlabel $-p-1$ at 113 111
\pinlabel $k$ at 75 72
\pinlabel $L_0$ at 145 103
\pinlabel $L_1$ at 30 56
\pinlabel $-1$ at 55 65
\pinlabel $-2$ at 66 63
\pinlabel $-p-1$ at 81 61
\pinlabel $L_0$ at 115 56
\pinlabel $L_1$ at 5 20
\pinlabel $-1$ at 23 29
\pinlabel $L_0$ at 71 20
\pinlabel $-p-1$ at 40 27
\pinlabel $L_1$ at 85 23
\pinlabel $-p$ at  113 24
\pinlabel $L_0$ at  142 23
\endlabellist     
\centering
\includegraphics[scale=1.5]{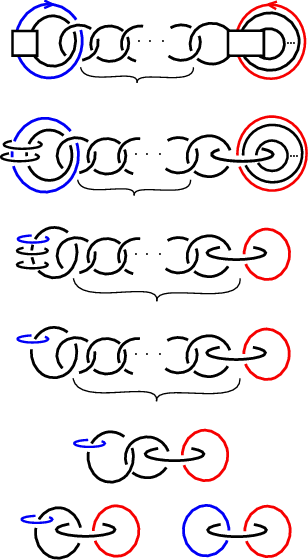}
\caption{Kirby moves for (c3).}
\label{figure:kirby-c3}
\end{figure}
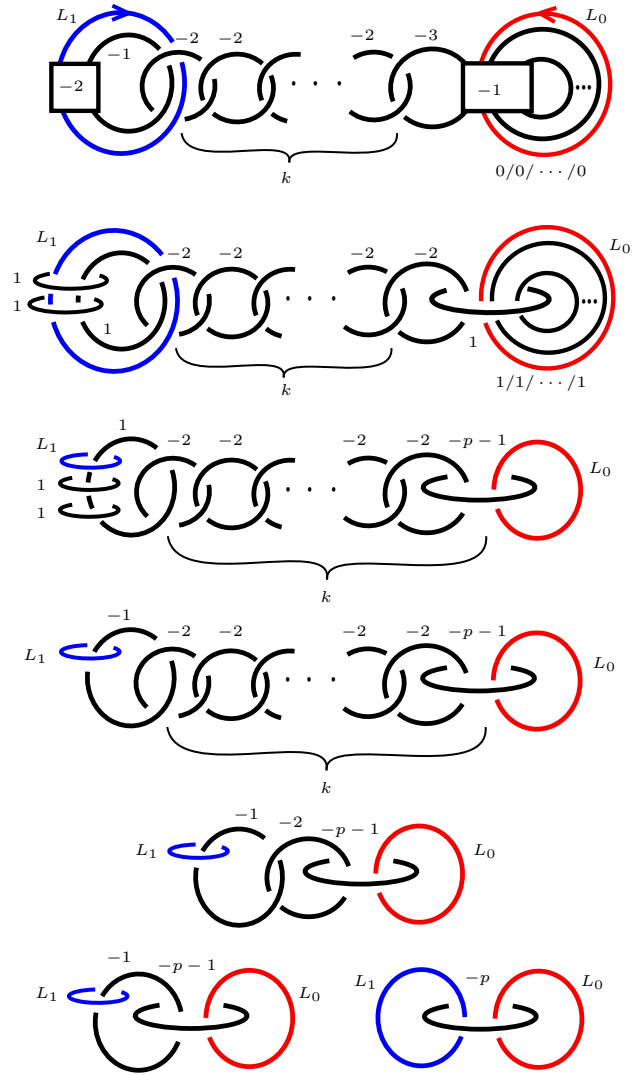

\begin{figure}[h]
\labellist
\tiny\hair 2pt
\pinlabel $L_1$ [br] at 8 125
\pinlabel $-p-2$ at 7.2 110
\pinlabel $-p-2$ at 21 119
\pinlabel $-1$ at 39 110
\pinlabel $0/0$ [t] at 55 91
\pinlabel $L_0$ [r] at 61 110
\pinlabel $L_1$ [br] at 99 125
\pinlabel $p+2$ [r] at 84.5 112
\pinlabel $\underbrace{1/1\cdots/1}_{p+2}$ [t] at 96 93
\pinlabel $1$ [b] at 112 101
\pinlabel $1$ [t] at 130 107
\pinlabel $1/1$ [t] at 150 91
\pinlabel $L_0$ [r] at 154 110
\pinlabel $L_1$ [br] at 16 66
\pinlabel $p+2$ [r] at 8 53
\pinlabel $\underbrace{1/1\cdots/1}_{p+2}$ [t] at 19 40
\pinlabel $1$ [bl] at 38 67
\pinlabel $1$ [t] at 47 50
\pinlabel $1/1$ [t] at 69 36
\pinlabel $L_0$ [r] at 72 55
\pinlabel $L_1$ [br] at 103 65
\pinlabel $-(p+1)$ [t] at 119 44
\pinlabel $-1$ [t] at 136 52
\pinlabel $L_0$ [bl] at 157 61
\pinlabel $L_1$ [br] at 65 21
\pinlabel $L_0$ [bl] at 115 21
\pinlabel $-p$ [t] at 90 11
\endlabellist
\centering
\includegraphics[scale=2]{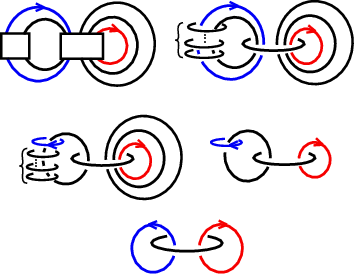}
\caption{Kirby moves for (d1).}
\label{figure:kirby-d1}
\end{figure}

\begin{figure}
\labellist
\tiny\hair 2pt
\pinlabel $L_1$ [br] at 13 265
\pinlabel $-1$ at 23 257
\pinlabel $-2$ at 11 248
\pinlabel $k$ [t] at 66 228
\pinlabel $-2$ at 40 260
\pinlabel $-2$ at 51 260
\pinlabel $-2$ at 83 260
\pinlabel $-p-3$ at 100 261
\pinlabel $-1$ at 117 248
\pinlabel $0/0$ [t] at 130 231
\pinlabel $L_0$ [bl] at 139 265
\pinlabel $L_1$ [br] at 10 209
\pinlabel $1$ at 24 203
\pinlabel $1$ at -3 193
\pinlabel $1$ at -3 200
\pinlabel $-2$ at 38 206
\pinlabel $-2$ at 53 206
\pinlabel $k$ [t] at 64 173
\pinlabel $1$ at 112 186
\pinlabel $-2$ at 83 207
\pinlabel $-p-2$ at 100 207
\pinlabel $1/1$ [t] at 130 177
\pinlabel $L_0$ [bl] at 142 209
\pinlabel $L_1$ [br] at 9 155
\pinlabel $1$ at 20 162
\pinlabel $1$ at 4 150
\pinlabel $1$ at 4 142
\pinlabel $-2$ at 38 158
\pinlabel $-2$ at 50 158
\pinlabel $k$ [t] at 75 124
\pinlabel $-p-1$ at 113 158
\pinlabel $-2$ at 83 159
\pinlabel $-2$ at 94 159
\pinlabel $L_0$ [bl] at 136 155
\pinlabel $L_1$ [br] at 9 107
\pinlabel $-1$ at 23 115
\pinlabel $-2$ at 36 110
\pinlabel $-2$ at 50 110
\pinlabel $k$ [t] at 75 76
\pinlabel $-2$ at 83 111
\pinlabel $-2$ at 95 111
\pinlabel $-p-1$ at 111 110
\pinlabel $L_0$ [bl] at 136 107
\pinlabel $L_1$ [br] at 36 57
\pinlabel $-1$ at 50 65
\pinlabel $-2$ at 67 63
\pinlabel $-p-1$ at 82 62
\pinlabel $L_0$ [bl] at 108 57
\pinlabel $L_1$ at 5 20
\pinlabel $-1$ at 23 28
\pinlabel $-p-1$ at 41 27
\pinlabel $L_0$ at 70 20 
\pinlabel $L_1$ at 85 24
\pinlabel $L_0$ at 142 24
\pinlabel $-p$ at 113 25   
\endlabellist
\centering
\includegraphics[scale=1.5]{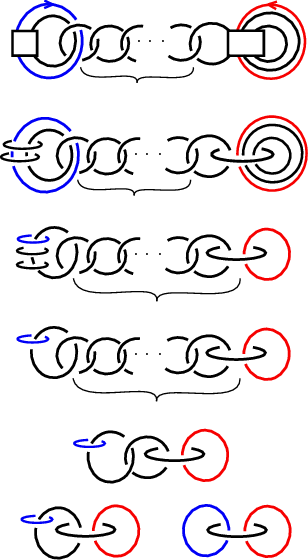}
\caption{Kirby moves for (d2).}
\label{figure:kirby-d2}
\end{figure}

\begin{figure}
\labellist
\tiny\hair 2pt
\pinlabel $L_1$ [br] at 7 214
\pinlabel $-2$ at 7 200
\pinlabel $-1$ at 20 209
\pinlabel $k_1$ [t] at 61 180
\pinlabel $k_0$ [t] at 129 180
\pinlabel $-2$ at 36 212
\pinlabel $-2$ at 49 212
\pinlabel $-2$ at 80 212
\pinlabel $-p-2$ at 97 212
\pinlabel $-2$ at 112 212
\pinlabel $-2$ at 144 213
\pinlabel $-1$ at 165 207
\pinlabel $-2$ at 179 197.5
\pinlabel $L_0$ [bl] at 177 214
\pinlabel $L_1$ [br] at 10 163
\pinlabel $1$ at 24 158
\pinlabel $1$ at -3 148
\pinlabel $1$ at -3 155
\pinlabel $k_1$ [t] at 61 129
\pinlabel $k_0$ [t] at 126 129
\pinlabel $-2$ at 37 160
\pinlabel $-2$ at 49 160
\pinlabel $-2$ at 80 160
\pinlabel $-p-2$ at 97 160
\pinlabel $-2$ at 112 160
\pinlabel $-2$ at 144 160
\pinlabel $1$ at 160 154
\pinlabel $L_0$ [bl] at 176 160
\pinlabel $1$ at 189 145
\pinlabel $1$ at 189 151
\pinlabel $L_1$ [br] at 7 107
\pinlabel $1$ at 24 114   
\pinlabel $1$ [r] at 6 100
\pinlabel $1$ [r] at 6 93
\pinlabel $k_1$ [t] at 61 79
\pinlabel $k_0$ [t] at 126 79
\pinlabel $-2$ at 36 110
\pinlabel $-2$ at 49 110
\pinlabel $-2$ at 80 110
\pinlabel $-p-2$ at 97 110
\pinlabel $-2$ at 112 110
\pinlabel $-2$ at 144 110
\pinlabel $1$ at 160 114
\pinlabel $L_0$ [bl] at 180 107
\pinlabel $1$ at 184 100
\pinlabel $1$ at 184 93
\pinlabel $L_1$ [br] at 9 62
\pinlabel $-1$ at 24 70
\pinlabel $k_1$ [t] at 62 35
\pinlabel $k_0$ [t] at 127 35
\pinlabel $-2$ at 36 66
\pinlabel $-2$ at 49 66
\pinlabel $-2$ at 80 66
\pinlabel $-p-2$ at 97 66
\pinlabel $-2$ at 112 66
\pinlabel $-2$ at 144 66
\pinlabel $L_0$ [bl] at 181 62
\pinlabel $-1$ at 160 70
\pinlabel $L_1$ [br] at 74 21 
\pinlabel $L_0$ [bl] at 122 21
\pinlabel $-p$ [t] at 97 10
\endlabellist
\centering
\includegraphics[scale=1.5]{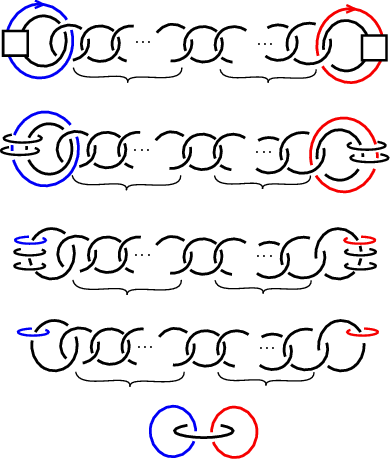}
\caption{Kirby moves for (d3).}
\label{figure:kirby-d3}
\end{figure}

\begin{figure}[h]
\labellist
\tiny\hair 2pt
\pinlabel $L_0$ [br] at 3 249
\pinlabel $-p-1$ [t] at 28 234
\pinlabel $k$ [t] at 68 218
\pinlabel $-2$ at 38 254
\pinlabel $-2$ at 51 254
\pinlabel $-2$ at 86 254
\pinlabel $-2$ at 100 254
\pinlabel $-1$ at 120 248
\pinlabel $-2$ at 134 239
\pinlabel $L_1$ [bl] at 133 255
\pinlabel $L_0$ [br] at 6 203
\pinlabel $-p-1$ [t] at 29 186
\pinlabel $k$ [t] at 70 170
\pinlabel $-2$ at 38 205
\pinlabel $-2$ at 51 205
\pinlabel $-2$ at 86 205
\pinlabel $-2$ at 100 205
\pinlabel $1$ [t] at 120 206
\pinlabel $L_1$ [bl] at 134 209
\pinlabel $1$ [l] at 146 200 
\pinlabel $1$ [l] at 146 194
\pinlabel $L_0$ [br] at 4 154
\pinlabel $-p-1$ [t] at 26 137
\pinlabel $k$ [t] at 67 122
\pinlabel $-2$ at 38 156
\pinlabel $-2$ at 51 156
\pinlabel $-2$ at 85 156
\pinlabel $-2$ at 100 156
\pinlabel $1$ [t] at 114 156
\pinlabel $L_1$ [bl] at 136 153
\pinlabel $1$ [l] at 136 147
\pinlabel $1$ [l] at 136 141
\pinlabel $L_0$ [br] at 10 104
\pinlabel $-p-1$ [t] at 33 89
\pinlabel $k$ [t] at 74 72
\pinlabel $-2$ at 40 107
\pinlabel $-2$ at 56 107
\pinlabel $-2$ at 90 107
\pinlabel $-2$ at 105 107
\pinlabel $-1$ [tl] at 132 89
\pinlabel $L_1$ [bl] at 142 104
\pinlabel $L_0$ [br] at 48 59
\pinlabel $L_1$ [bl] at 103 59
\pinlabel $-p-1$ [t] at 71 43
\pinlabel $-1$ [tl] at 94 45
\pinlabel $L_0$ [br] at 48 22
\pinlabel $-p$ [t] at 71 9
\pinlabel $L_1$ [bl] at 95 22
\endlabellist
\centering
\includegraphics[scale=1.5]{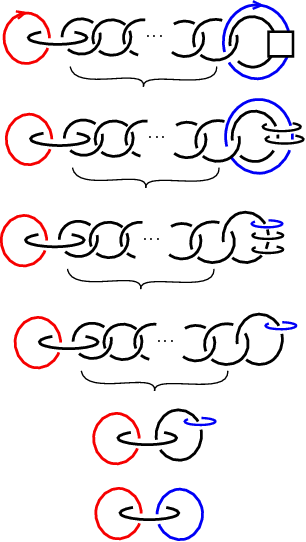}
\caption{Kirby moves for (e2).}
\label{figure:kirby-e2}
\end{figure}
\subsection{Legendrian realisations} 
\subsubsection{Case (b)}
The $|\ttt_1-1|$ realisations with $\ttt_0=0$ and $\ttt_1\leq 0$
are shown in Figure~\ref{figure:case-b}. Here $k$ and $\ell$ denote the
number of stabilisations, with $k+\ell=-\ttt_1$, so that
$\tb_1=\ttt_1-1$.
 
\begin{figure}[h]
\labellist
\small\hair 2pt
\pinlabel $L_1$ [bl] at 57 110
\pinlabel $k$ [r] at 1 86
\pinlabel $\ell$ [l] at 71 86
\pinlabel $+1$ [r] at 0 59
\pinlabel $+1$ [r] at 0 51
\pinlabel $+1$ [r] at 0 36
\pinlabel $+1$ [r] at 0 29
\pinlabel $p+1$ [l] at 77 46
\pinlabel $L_0$ [tl] at 53 12
\endlabellist
\centering
\includegraphics[scale=1.1]{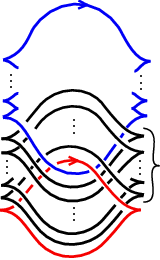}
\caption{$|\ttt_1-1| $ Legendrian realisations for case (b).}
\label{figure:case-b}
\end{figure}

The linking matrix $M$ of the surgery diagram is the
$((p+1)\times (p+1))$-matrix
\[ \begin{pmatrix} 
0      & -1     & -1     & \cdots & -1\\
-1     & 0      & -1     & \cdots & -1\\
-1     & -1     & 0      & \cdots & -1\\
\vdots & \vdots & \vdots & \ddots & \vdots\\
-1     & -1     & -1     & \cdots & 0
\end{pmatrix}. \]
The determinant of this matrix is $\det M=-p$.
The extended linking matrix for $L_i$, $i=0,1$, is the
$((p+2)\times (p+2))$-matrix built in the same fashion as~$M$, so that
$\det(M_i) =-p-1$.

It follows that 
\[ \tb_{\Q}(L_0)=-1+\frac{p+1}{p}=\frac{1}{p}\]
and
\[ \tb_{\Q}(L_1)=\ttt_1-1+\frac{p+1}{p}=\ttt_1+\frac{1}{p}.\]
Since $\underline{\rot}=\underline{0}$, we get
$\rot_{\Q}(L_0)=0$ and
\[ \rot_{\Q}(L_1)\in \{\ttt_1, \ttt_1+2,\ldots,-\ttt_1-2,-\ttt_1\}.\]
 
For the calculation of $d_3$, we observe that $c^2=0$, $\chi=p+2$,
and $\sigma=p-1$. Thus, 
\[d_3=\frac{1}{4}(0-2(p+2)-3(p-1))+p+1=\frac{1}{4}(-5p-1)+p+1=\frac{3-p}{4}.\]
 
By Theorem~\ref{thm:geon15}, there are no exceptional realisations
of a rational unknot in $L(p,1)$ with $\tb_{\Q}$ equal to that of~$L_1$,
so $L_1$ is loose. The component $L_0$ is exceptional, as can
be seen by performing surgery on it. A contact $(-\frac{1}{p+2})$-surgery
on $L_0$ has the same effect as taking $p+2$ Legendrian push-offs of $L_0$
and doing a $(-1)$-surgery on each of them. This cancels the
$(+1)$-surgeries in the diagram and hence produces a tight contact
$3$-manifold.
\subsubsection{Case (c2)}
In this case we have exactly three Legendrian realisations.
The left-hand side of Figure \ref{figure:case-c2}, where the
rational rotation numbers of $L_0$ and $L_1$ will be shown to be
non-zero, gives two realisations (one with $L_0,L_1$ oriented as shown,
the second with orientations flipped simultaneously.
The right-hand side, where the rotation numbers are
zero, gives the third one.

\begin{figure}[h]
 \labellist
\small\hair 2pt
\pinlabel $L_1$ [bl] at 58 99
\pinlabel $+1$ [l] at 73 80
\pinlabel $+1$ [r] at 2 56
\pinlabel $+1$ [r] at 2 49
\pinlabel $+1$ [r] at 2 35
\pinlabel $+1$ [r] at 2 28
\pinlabel $p+2$ [l] at 78 44
\pinlabel $L_0$ [tl] at 58 15
\pinlabel $L_1$ [bl] at 172 99
\pinlabel $+1$ [l] at 189 80
\pinlabel $+1$ [r] at 114 56 
\pinlabel $+1$ [r] at 114 49
\pinlabel $+1$ [r] at 114 35
\pinlabel $+1$ [r] at 114 28
\pinlabel $p+2$ [l] at 190 44
\pinlabel $L_0$ [tl] at 172 15
\endlabellist
\centering
\includegraphics[scale=1.1]{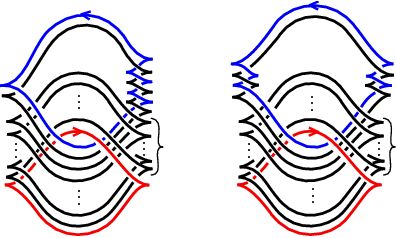}
\caption{Three Legendrian realisations for case (c2).}
\label{figure:case-c2}
\end{figure}

We begin with the left-hand side.
The linking matrix for the surgery diagram, ordering the surgery knots
from top to bottom, and with all surgery knots oriented clockwise,
is the $((p+3)\times (p+3))$-matrix
\[ M=\begin{pmatrix} 
-2 & -1 & -1 & \cdots & -1\\
-1 & 0  & -1 & \cdots & -1\\
-1 & -1 & 0  & \cdots & -1\\
\vdots & \vdots & \vdots & \ddots & \vdots\\
-1 & -1 & -1 & \cdots & 0
\end{pmatrix},\]
with $\det{M}=p$.
 
The extended linking matrices for $L_0$ and $L_1$ are
\[ M_0= \left(\begin{array}{c|cccc}
0      & -1 & -1 & \cdots & -1\\ \hline
-1     &    &    &        &   \\
-1     &    & M  &        &   \\
\vdots &    &    &        &   \\
-1     &    &    &        &   
\end{array}\right),\]
with $\det{M_0}=1+p$, and
\[ M_1= \left(\begin{array}{c|cccc} 
0      & 3 & 1 & \cdots & 1\\ \hline
3      &   &   &        &  \\
1      &   & M &        &  \\
\vdots &   &   &        &  \\
1      &   &   &        &
\end{array}\right),\]
with $\det{M_1}=1+5p$. This yields
\[ \tb_{\Q}(L_0)=-1+\frac{1+p}{p}=\frac{1}{p}\;\;\;\text{and}\;\;\;
\tb_{\Q}(L_1)=-3+\frac{1+5p}{p}=2+\frac{1}{p},\] 
so we are indeed in the case $(\ttt_0,\ttt_1)=(0,2)$.
Further, 
\[ \underline{\rot}=(2,\underbrace{0,0,\ldots, 0}_{p+2}),\;\;
\vlk_0=(-1,\underbrace{-1,\ldots, -1}_{p+2}),\;\;
\vlk_1= (3,\underbrace{1,\ldots,1}_{p+2}).\]
A simple calculation gives
\[ M^{-1}\vlk_0=(-1/p,1/p,\ldots,1/p)\;\;\;\text{and}\;\;\;
M^{-1}\vlk_1=(-(2p+1)/p,1/p,\ldots,1/p).\]
Hence
\[ \rot_{\Q}(L_0)=0-2\cdot\frac{-1}{p}=\frac{2}{p}\;\;\;\text{and}\;\;\;
\rot_{\Q}(L_1)=-2+2\cdot\frac{2p+1}{p}=2+\frac{2}{p}.\]
The second realisation is given by reversing the orientations of
both $L_0$ and $L_1$; this changes the sign of the $\rot_{\Q}(L_i)$
(as is best seen by also changing the orientations of the surgery curves,
although the choice of orientation here does not matter).

For the diagram on the right, the calculation of $\tb_{\Q}(L_i)$
does not change, but now we have $\vrot=\underline{0}$ and $\rot_1=0$,
whence $\rot_{\Q}(L_0)=\rot_{\Q}(L_1)=0$. This yields the third realisation.

For the calculation of the $d_3$-invariant we observe that both diagrams
in Figure~\ref{figure:case-c2} give $\sigma=p-1$ and $\chi=p+4$.
The diagram on the right has $\vrot=\underline{0}$, which yields
$d_3=(7-p)/4$. For the diagram on the left, 
the solution $\bfx$ of $M\bfx=\vrot$ is $\bfx=(-2-2/p,2/p,\ldots,2/p)^{\ttt}$.
Then $c^2=\langle\bfx,\vrot\rangle=-2(2+2/p)$ and $d_3=
(3p-p^2-4)/4p$.

In either case, the contact structure is overtwisted by the observations in
Section~\ref{section:detecting}.
A contact $(-1)$-surgery along $L_1$ and a contact $(-\frac{1}{p+2})$-surgery
along $L_0$ cancels all $(+1)$-surgeries, so the link is exceptional.
Both components are loose, since the classical invariants do not
match those of Theorem~\ref{thm:geon15}.
\subsubsection*{Case (c3)}
Two of the four Legendrian realisations in this case are shown in
Figure~\ref{figure:case-c3}, with $L_0$ and $L_1$ both oriented
clockwise or counter-clockwise for $k$ even, $L_1$ having the opposite
orientation of $L_0$ for $k$ odd; cf.\ Figure~\ref{figure:kirby-c3}.
The other two are obtained by
flipping the shark with surgery coefficient $-1$.
The invariants are shown in
Table~\ref{table:c3}. We omit the calculations of the invariants
in this and the remaining cases. The arithmetic is lengthy but
not inspiring; the detailed calculations are available from
the authors upon request.

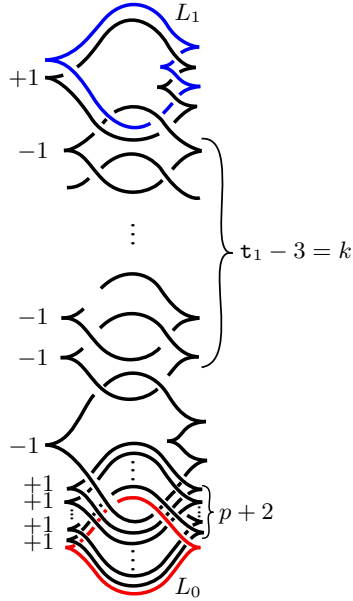
\begin{figure}[h]
\labellist
\small\hair 2pt
\pinlabel $L_1$ [bl] at 43 196
\pinlabel $+1$ [r] at 0 178
\pinlabel $\ttt_1-3=k$ [l] at 65 119
\pinlabel $-1$ [r] at 3 153
\pinlabel $-1$ [r] at 3 96
\pinlabel $-1$ [r] at 3 82
\pinlabel $-1$ [r] at 0 52
\pinlabel $+1$ [r] at 5 39
\pinlabel $+1$ [r] at 5 33
\pinlabel $+1$ [r] at 5 25
\pinlabel $+1$ [r] at 5 19
\pinlabel $p+2$ at 69 29
\pinlabel $L_0$ [tl] at 43 8
\endlabellist
\centering
\includegraphics[scale=1.1]{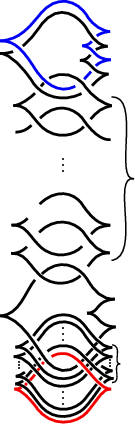}
\caption{Legendrian realisations for case (c3).}
\label{figure:case-c3}
\end{figure}

\begin{table}
{\renewcommand{\arraystretch}{1.5}
\begin{tabular}{|c|c|c|c|c|} \hline
$\ttt_0$ & $\rot_{\Q}(L_0)$ & $\ttt_1$ & $\rot_{\Q}(L_1)$ & $d_3$ \\ 
\hline\hline
$0$ & $0$ & $> 2$ & $\pm(\ttt_1-2)$  & $\frac{7-p}{4}$\\[.5mm] 
\hline
$0$ & $\pm\frac{2}{p}$ & $> 2$ & $\pm(\ttt_1+\frac{2}{p})$  &
  $\frac{3p-p^2-4}{4p}$ \\[.5mm] 
\hline
\end{tabular}
} % end arraystretch
\vspace{1.5mm}
\caption{Invariants for case (c3).}
\label{table:c3}
\end{table}

The contact structure is overtwisted, because we have rational unknots
with $\ttt_i\geq 0$. A contact $(-1)$-surgery on $L_1$ and a contact
$(-\frac{1}{p+2})$-surgery on $L_0$ will cancel all contact $(+1)$-surgeries,
so the link is exceptional.

Except for $(\ttt_0,\rot_{\Q}(L_0))=(0,0)$,
the invariants of $L_0$ and $L_1$ are not realised by
exceptional rational unknots, see Theorem~\ref{thm:geon15} or,
for a better overview, \cite[Table~2]{geon15}; hence these knots are loose.
So is $L_0$ in the first line of Table~\ref{table:c3}, since
the value of the $d_3$-invariant does not match that of an exceptional
realisation; again see \cite[Table~2]{geon15}. This is the first
instance where we need the $d_3$-invariant to detect looseness.
\subsubsection{Case (d1)}
The $p+3$ Legendrian realisations with invariants as listed
in Theorem~\ref{thm:main} are shown in Figure~\ref{figure:case-d1}.
The surgery knot of which $L_1$ is a push-off is a $(p+2)$-fold
stabilisation of the standard Legendrian unknot (compare with
the topological surgery framing in Figure~\ref{figure:kirby-d1});
placing the stabilisations left or right gives the $p+3$ choices.
Flipping the orientations of both $L_0$ and $L_1$ has the same
effect as exchanging the number of stabilisations on the left-
and right-hand side, so this does not give any new choices.

\begin{figure}[h]
\labellist
\small\hair 2pt
\pinlabel $L_1$ [bl] at 55 95
\pinlabel $+1$ [r] at 0 78
\pinlabel $+1$ [r] at 0 37
\pinlabel $+1$ [r] at 0 28
\pinlabel $L_0$ [tl] at 55 14
\endlabellist
\centering
\includegraphics[scale=1]{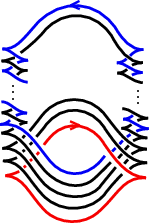}
\caption{$p+3$ Legendrian realisations for case (d1).}
\label{figure:case-d1}
\end{figure}
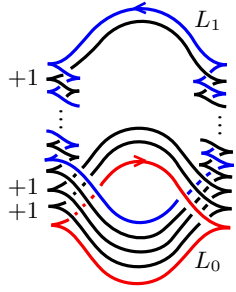

The argument for the link being exceptional is as in the preceding case.
According to \cite[Table~2]{geon15},
the classical invariants are realised by exceptional rational unknots,
but only in the overtwisted contact structure with $d_3=(3p-\ttr^2)/4p$,
so the link components here are loose.
\subsubsection{Case (d2)}
The $2(p+2)$ Legendrian realisations are shown in Figure~\ref{figure:case-d2},
with $L_0,L_1$ given the same orientation for $k$ even, the opposite, for
$k$ odd. The factor $2$ comes from simultaneously exchanging
the orientations of $L_0$ and~$L_1$.
The factor $p+2$ comes from the placement of $p+1$
stabilisations on the surgery knot near the bottom of the diagram
(compare with Figure~\ref{figure:kirby-d2}). We write
\[ \ttr\in\{-p-1,-p+1,\ldots, p-1,p+1\}\]
for the $p+2$ possible rotation numbers of this surgery knot.
The invariants are listed in Table~\ref{table:d2}.
The remaining arguments are as before.

\begin{figure}[h]
\labellist
\small\hair 2pt
\pinlabel $L_1$ [bl] at 54 214
\pinlabel $+1$ [l] at 65 199
\pinlabel $-1$ [l] at 65 170
\pinlabel $-1$ [l] at 65 113
\pinlabel $-1$ [l] at 65 100
\pinlabel $-1$ [l] at 78 74
\pinlabel $+1$ [l] at 78 37
\pinlabel $+1$ [l] at 78 29
\pinlabel $k=\ttt_1-2$ [r] at 0 135
\pinlabel $L_0$ [tl] at 54 6
\endlabellist
\centering
\includegraphics[scale=1]{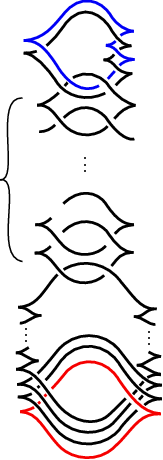}
\caption{$2(p+2)$ Legendrian realisations for case (d2).}
\label{figure:case-d2}
\end{figure}
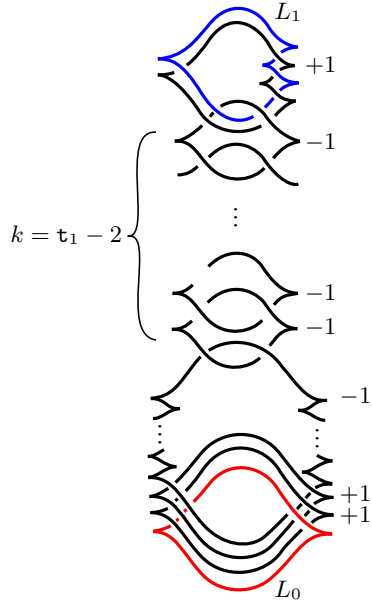

\begin{table}
{\renewcommand{\arraystretch}{1.5}
\begin{tabular}{|c|c|c|c|c|} \hline
$\ttt_0$ & $\rot_{\Q}(L_0)$ & $\ttt_1$ & $\rot_{\Q}(L_1)$ & $d_3$ \\
\hline\hline
$1$ & $\pm\frac{1-\ttr}{p}$ & $>1$ &
$\pm(\ttt_1-1+\frac{1-\ttr}{p})$ &
$\frac{7p-1+2\ttr-\ttr^2}{4p}$ \\[.5mm] 
\hline
\end{tabular}
} % end arraystretch
\vspace{1.5mm}
\caption{Invariants for case (d2).}
\label{table:d2}
\end{table}
\subsubsection{Case (d3)}
The $4(p+1)$ Legendrian realisations are shown in Figure~\ref{figure:case-d3}.
For $k_0,k_1$ of equal parity, we give $L_0$ and $L_1$ the same
orientation; for $k_0,k_1$ of opposite parity, the opposite
orientation. A factor $2$ comes from exchanging the orientations
simultaneously, a factor $2$ from the two diagrams, and a factor $p+1$
from placing $p$ stabilisations left or right on the surgery knot
at the centre of each diagram, cf.\ Figure~\ref{figure:kirby-d3}.
The rotation number of this surgery knot is denoted by
\[ \ttr\in\{ -p, -p+2,\ldots, p-2,p\}.\]
The invariants are shown in Table~\ref{table:d3}.
As before, suitable contact surgeries along $L_0,L_1$
and comparison with \cite{geon15} show that the link
is exceptional and the components loose.

\begin{figure}[h]
\labellist
\small\hair 2pt
\pinlabel $L_1$ [bl] at 48 314
\pinlabel $+1$ [l] at 60 284 
\pinlabel $-1$ [l] at 60 270
\pinlabel $k_1=\ttt_1-2$ [r] at 0 234
\pinlabel $-1$ [l] at 60 212
\pinlabel $-1$ [l] at 60 199
\pinlabel $-1$ [bl] at 56 181
\pinlabel $-1$ [l] at 62 123
\pinlabel $k_0=\ttt_0-2$ [r] at 0 88
\pinlabel $-1$ [l] at 62 67
\pinlabel $-1$ [l] at 62 53
\pinlabel $L_0$ [bl] at 24 16
\pinlabel $+1$ [l] at 60 13
\pinlabel $L_1$ [bl] at 145 314
\pinlabel $+1$ [r] at 110 284
\pinlabel $-1$ [r] at 110 270
\pinlabel $\ttt_1-2=k_1$ [l] at 175 234
\pinlabel $-1$ [r] at 110 212
\pinlabel $-1$ [r] at 110 199
\pinlabel $-1$ [br] at 118 181
\pinlabel $-1$ [r] at 112 123
\pinlabel $\ttt_0-2=k_0$ [l] at 175 88
\pinlabel $-1$ [r] at 112 67
\pinlabel $-1$ [r] at 112 53
\pinlabel $L_0$ [bl] at 127 16
\pinlabel $+1$ [r] at 107 23
\endlabellist
\centering
\includegraphics[scale=1]{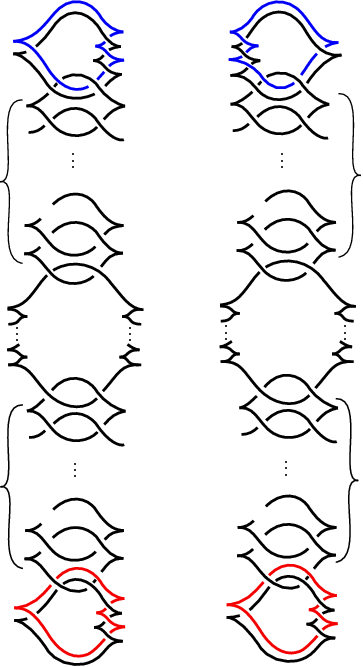}
\caption{$4(p+1)$ Legendrian realisations for case (d3).}
\label{figure:case-d3}
\end{figure}

\begin{table}
{\renewcommand{\arraystretch}{1.5}
\begin{tabular}{|c|c|c|c|c|} \hline
$\ttt_0$ & $\rot_{\Q}(L_0)$ & $\ttt_1$ & $\rot_{\Q}(L_1)$ & $d_3$ \\
\hline\hline
$>1$ & $\pm(\ttt_0-1+\frac{2-\ttr}{p})$ & $>1$ &
     $\pm(\ttt_1-1+\frac{2-\ttr}{p})$ & $\frac{7p-4+4\ttr-\ttr^2}{4p}$\\[.5mm] 
\hline
$>1$ & $\pm(\ttt_0-1+\frac{\ttr}{p})$   & $>1$ &
     $\mp(\ttt_1-1-\frac{\ttr}{p})$   & $\frac{7p-\ttr^2}{4p}$\\[.5mm]
\hline
\end{tabular}
} % end arraystretch
\vspace{1.5mm}
\caption{Invariants for case (d3).}
\label{table:d3}
\end{table}
\subsubsection{Case (e1)}
Figure~\ref{figure:case-e1} shows $|\tb_0|(p+1)$ Legendrian realisations.
Placing the $|\tb_0|-1$ stabilisations left or right gives $|\tb_0|$ choices.
The surgery knot is the standard Legendrian unknot with $p$
stabilisations, which gives $p+1$ choices. Flipping the orientations
of $L_0,L_1$ simultaneously is the same as flipping the whole
diagram (and the left/right stabilisations), so this does not
give any additional realisations. Write
\[ \ttr\in\{-p,-p+2,\ldots, p-2,p\}\]
for the rotation number of the surgery knot. A straightforward calculation
then gives $\ttt_0=\tb_0$ (and hence the right number
of realisations) and the other classical invariants as listed
in Theorem~\ref{thm:main}, with $\ttr_0:=\rot_0$.
The $d_3$-invariant takes the value $(3p-\ttr^2)/4p$.
Since $\ttt_1>0$, the contact structure is overtwisted.
A contact $(-1)$-surgery on $L_1$ gives $(S^3,\xist)$, so $L_1$ is
exceptional. On the other hand, $L_0$ must be loose, since $\ttt_0<0$.

\begin{figure}[h]
\labellist
\small\hair 2pt
\pinlabel $L_1$ [bl] at 56 123
\pinlabel $+1$ [r] at 0 73
\pinlabel $L_0$ [tl] at 56 11
\endlabellist
\centering
\includegraphics[scale=1.1]{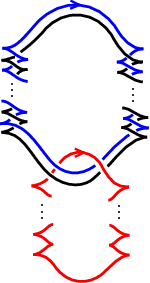}
\caption{$|\ttt_0|(p+1)$ Legendrian realisations for case (e1).}
\label{figure:case-e1}
\end{figure}
\subsubsection{Case (e2)}
\label{subsubsection:e2}
The $2|\ttt_0|p$ realisations for this case are shown in
Figure~\ref{figure:case-e2}, with $L_0$ and $L_1$ both oriented
clockwise or counter-clockwise for $k$ odd, $L_1$ having the opposite
orientation of $L_0$ for $k$ even; cf.\ Figure~\ref{figure:kirby-e2}.
Here $L_0$ is an unknot with
$\tb_0<0$ (which will again turn out to equal~$\ttt_0$);
this gives us $|\ttt_0|$ choices distinguished by
\[ \ttr_0:=\rot_0\in\{\ttt_0+1,\ttt_0+3,\ldots,-\ttt_0-3,-\ttt_0-1\}.\]
The topmost surgery knot has Thurston--Bennequin invariant $-p$;
this gives us $p$ choices distinguished by the rotation number
\[ \ttr\in\{-p+1,-p+3,\ldots,p-3,p-1\}.\]
The factor $2$ in the number of choices comes from simultaneously
flipping the orientations of $L_0$ and $L_1$.

One then computes that
\[ \rot_{\Q}(L_0)=\pm\Bigl(\ttr_0+\frac{\ttr+1}{p}\Bigr),\;\;\;
\rot_{\Q}(L_1)=\pm\Bigl(\ttt_1-1+\frac{\ttr+1}{p}\Bigr),\]
and
\[ d_3=\frac{1}{4}\Bigl(3+\frac{2\ttr-\ttr^2-1}{p}\Bigr).\]
The argument for showing that $L_0$ is loose, and $L_1$ exceptional,
is as in case (e1).

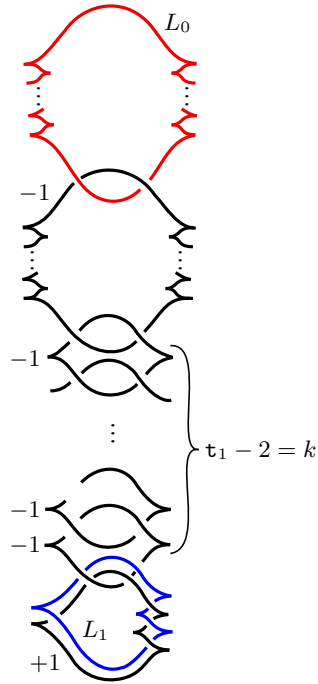
\begin{figure}[h]
\labellist
\small\hair 2pt
\pinlabel $L_0$ [bl] at 52 244
\pinlabel $-1$ [br] at 12 180
\pinlabel $\ttt_1-2=k$ [l] at 67 88
\pinlabel $-1$ [r] at 9 122
\pinlabel $-1$ [r] at 9 65
\pinlabel $-1$ [r] at 9 52
\pinlabel $L_1$ [bl] at 21 15
\pinlabel $+1$ [tr] at 16 12
\endlabellist
\centering
\includegraphics[scale=1]{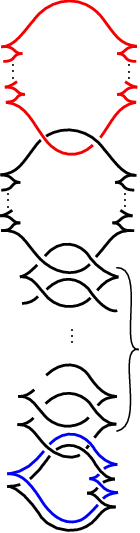}
\caption{$2|\ttt_0|p$ Legendrian realisations for case (e2).}
\label{figure:case-e2}
\end{figure}


\begin{thebibliography}{10}
%
\bibitem{baet12}
\textsc{K. Baker and J. Etnyre},
Rational linking and contact geometry,
\textit{Perspectives in Analysis, Geometry, and Topology},
Progr. Math. \textbf{296}
(Birkh\"auser Verlag, Basel, 2012), 19--37.
%
\bibitem{chke}
\textsc{R. Chatterjee} and \textsc{M. Kegel},
Contact surgery numbers of $\Sigma(2,3,11)$ and $L(4m+3,4)$,
\texttt{arXiv:2404.18177}.
%
\bibitem{dige04}
\textsc{F. Ding and H. Geiges},
A Legendrian surgery presentation of contact $3$-manifolds,
\textit{Math. Proc. Cambridge Philos. Soc.}
\textbf{136} (2004), 583--598.
%
\bibitem{dgs04}
\textsc{F. Ding, H. Geiges and A. I. Stipsicz},
Surgery diagrams for contact $3$-manifolds,
\textit{Turkish J. Math.}
\textbf{28} (2004), 41--74.
%
\bibitem{geig08}
\textsc{H. Geiges},
\textit{An Introduction to Contact Topology},
Cambridge Stud. Adv. Math. \textbf{109}
(Cambridge University Press, Cambridge, 2008).
%
\bibitem{geon15}
\textsc{H. Geiges and S. Onaran},
Legendrian rational unknots in lens spaces,
\textit{J. Symplectic Geom.}
\textbf{13} (2015), 17--50.
%
\bibitem{geon20}
\textsc{H. Geiges and S. Onaran},
Legendrian Hopf links,
\textit{Quart. J. Math.} \textbf{71} (2020), 1419--1459.
%
\bibitem{giro00}
\textsc{E. Giroux},
Structures de contact en dimension trois et bifurcations des
feuilletages de surfaces,
\textit{Invent. Math.}
\textbf{141} (2000), 615--689.
%
\bibitem{hond00I}
\textsc{K. Honda},
On the classification of tight contact structures~I,
\textit{Geom. Topol.}
\textbf{4} (2000), 309--368;
erratum: \textit{Geom.\ Topol.}
\textbf{5} (2001), 925--938.
%
\bibitem{lerm01}
\textsc{E. Lerman},
Contact cuts,
\textit{Israel J. Math.}
\textbf{124} (2001), 77--92.
%
\end{thebibliography}
\end{document}